\documentclass[11pt]{article}
\usepackage{amsthm}
\usepackage{amsmath}
\usepackage{amsfonts}
\usepackage{amssymb}
\usepackage{colordvi} 
\usepackage{graphicx} 
\usepackage{vmargin} 
\usepackage{mathrsfs}
\usepackage{url}
\usepackage[all]{xy}
\date{}

\newtheorem{thm}{Theorem}[section]
\newtheorem{cor}[thm]{Corollary}
\newtheorem{lem}[thm]{Lemma}
\newtheorem{prop}[thm]{Proposition}

\newtheorem{ex}[thm]{Example}
\theoremstyle{definition}
\newtheorem{defn}[thm]{Definition}
\newtheorem{example}[thm]{Example}
\newtheorem{remark}[thm]{Remark}

\mathchardef\ordinarycolon\mathcode`\:
\mathcode`\:=\string"8000
\begingroup \catcode`\:=\active
\gdef:{\mathrel{\mathop\ordinarycolon}}
\endgroup

\newcommand{\R}{\mathbb{R}}
\newcommand{\C}{\mathbb{C}}
\newcommand{\Z}{\mathbb{Z}}

\newcommand{\Q}{\mathbb{Q}}
\newcommand{\PP}{\mathbb{P}}

\newcommand{\M}{\mathcal{M}}
\newcommand{\sO}{\mathscr{O}}
\newcommand{\G}{\mathbb{G}}

\newcommand{\dist}{\mathop{\mathrm{dist}}}
\newcommand{\Spec}{\mathop{\mathrm{Spec}}}

\newcommand{\initial}{\mathop{\mathrm{in}}}

\let\oldS=\S
\def\S{\oldS\,}

\title{\bf Mustafin Varieties}

\author{Dustin Cartwright, Mathias H\"abich, \\
Bernd Sturmfels and Annette Werner}

\begin{document}
\maketitle

\begin{abstract} \noindent
A Mustafin variety is a degeneration of projective space
induced by a point configuration in a Bruhat-Tits building.
The special fiber is reduced and Cohen-Macaulay, and its
irreducible components form interesting combinatorial patterns.
For configurations that lie in one apartment, these patterns are
regular mixed subdivisions of scaled simplices, and
the Mustafin variety is a twisted Veronese variety 
built from such a subdivision. This connects
our study to  tropical and toric geometry.   
For general configurations, the irreducible components of the special fiber
are rational varieties, and any blow-up of
projective space along a linear subspace arrangement
can arise. A detailed study of Mustafin varieties is undertaken
for  configurations in the Bruhat-Tits tree of $ PGL(2)$
and in the two-dimensional building of $PGL(3)$.
The latter yields the classification of 
Mustafin triangles into $38$ combinatorial types.
\end{abstract}

%\maketitle

\section{Introduction}

This paper introduces a novel combinatorial theory of degenerations of projective spaces.
Our degenerations are induced by $n$-tuples of $d \times d$-matrices over a field with a valuation,
and they are entirely natural from the perspectives of linear algebra, tropical geometry, 
and computational algebra. When the matrices are diagonal matrices then we recover
mixed subdivisions of scaled simplices, delightful structures that are known to be equivalent to 
tropical polytopes  and to  triangulations of products of simplices. Our aim here is
to develop the non-abelian theory, where 
the given matrices are no longer diagonal. The combinatorial implications of this are illustrated
in Figure~\ref{fig:planar}, where the left diagram shows the familiar abelian case while the right picture
shows the non-abelian case.

The total spaces in our degenerations are called Mustafin varieties, 
and the combinatorial objects referred to above are their special fibers.
Degenerations are a central topic in arithmetic geometry. The projective space
plays an important role here since any of its degenerations
induces a degeneration of every projective subvariety.

We now present our algebraic set-up in  precise terms.
Let $K$ be a field with a discrete valuation $v\colon K^* \rightarrow \Z$, and let $R$ be its ring of integers and $k$ its residue field. 
For example, $K$ could be the field of rational functions $k(t)$ or the field of formal Laurent series $k(\!(t)\!)$ over any ground field $k$, or it could be the field $\Q_p$ of $p$-adic numbers for some prime number $p$. We fix a prime element $\pi$ in the ring of integers $R$, i.e.\ $\pi$ is an element of the field $K$ having minimal positive valuation.

Let $V$ be a vector space of dimension $d \geq 2$ over $K$ and denote by $\PP(V)= \operatorname{Proj \: Sym} V^\ast$ the corresponding projective space, where $V^\ast$ is the dual space of $V$.
The projective space $\PP(V)$ parametrizes lines through the origin in $V$. 
We regard $V$ as an $R$-module,  and
a \emph{lattice} in $V$ is any $R$-submodule $L \subset V$
that is free of rank $d$.
If $L$ is a lattice in $V$, we denote by $\PP(L) = \operatorname{Proj \: Sym} L^\ast$ the corresponding projective space over the ring of integers $R$. Here, 
$L^\ast = {\rm Hom}_R(L,R)$ denotes the dual $R$-module.

\begin{defn}
Let $\Gamma =\{L_1, \ldots, L_n \}$ be a  set of lattices in $V$.
Then $\PP(L_1), \ldots, \PP(L_n)$ are projective spaces over $R$
whose generic fibers are canonically isomorphic to the projective space
$\PP(V) \simeq \PP_K^{d-1}$.
 The open immersions $\PP(V) \hookrightarrow \PP(L_i)$ give rise to
 a map
 \begin{equation*}
\PP(V)\, \longrightarrow \,\PP(L_1) \times_R \ldots \times_R \PP(L_n).
\end{equation*}
 Let $\M(\Gamma)$ be the closure of the image endowed with the reduced 
 scheme structure. We call  $\M(\Gamma)$ the \emph{Mustafin variety} associated to
 the set of lattices~$\Gamma$. Note that
 $\M(\Gamma)$ is a scheme over $R$ whose generic fiber is
$\PP(V)$. Its special fiber $\M(\Gamma)_k$ is a scheme over $k$.
\end{defn}
 
The construction of the Mustafin variety $\M(\Gamma)$ depends
only on the homothety classes of the lattices $L_i$, so throughout this
paper we regard  $\Gamma$ as a configuration  in the Bruhat-Tits building
$\mathfrak{B}_d$ associated with the group $PGL(V)$.
Varieties of the form $\M(\Gamma)$ were investigated by Mustafin \cite{mus}
 in order to generalize Mumford's seminal work \cite{mum} on uniformization of curves to higher dimension. Mustafin primarily considered the case of convex subsets $\Gamma$, as defined 
 in the text prior to Theorem~\ref{sec:must;thm:convex}.
 
In the present paper we are interested in arbitrary configurations $\Gamma$.
 The resulting Mustafin varieties have worse singularities
  but their combinatorial structure is richer.
 We note that every Mustafin variety is dominated by one from a convex configuration. 
 Indeed, any inclusion $\Gamma \subset \Gamma' $
specifies a surjective morphism $\M(\Gamma') \rightarrow \M(\Gamma)$
and the set of lattice points in the convex hull 
of~$\Gamma$  is a finite set, as
seen in \cite{fa, JSY}.

The term ``Mustafin variety'' is used here for the first time. In the previous discussions
of  these objects in \cite{CS, fa, KT}, Mustafin
varieties had been called ``Deligne schemes'', since for a convex set of vertices the
corresponding Mustafin variety represents the so-called Deligne functor. 
 We decided to name them after
G.\,A.~Mustafin, to recognize the contributions of \cite{mus},
and we opted for ``variety'' over ``scheme'' because $\M(\Gamma)$ and 
its special fiber $\M(\Gamma)_k$ are reduced for all lattice
configurations~$\Gamma \subset \mathfrak{B}_d$.

Figure \ref{fig:planar} shows two pictures representing Mustafin
varieties for  $d=3$ and $n=4$. In both cases, the special fiber
$\M(\Gamma)_k$ is a surface with ten irreducible components,
namely four copies of $\PP_k^2$ and six copies of
$\PP_k^1 \times \PP^1_k$. The left picture is {\em planar}
(it is a {\em regular mixed subdivision}) because the configuration $\Gamma$ 
lies in a single apartment of $\mathfrak{B}_3$, while the
right picture represents a {\em non-planar} case when 
$\Gamma$ is not in one apartment.

%\begin{center}
\begin{figure}
\centering
\includegraphics[width=5.9cm]{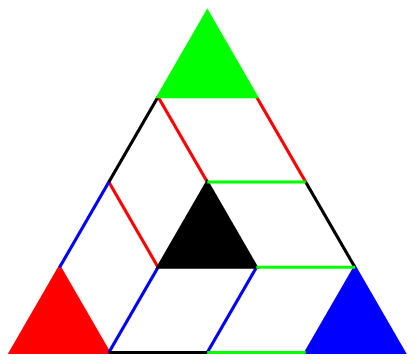} \qquad \quad
\includegraphics[width=5.7cm]{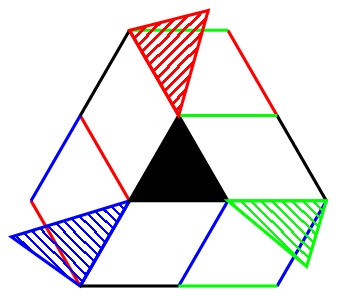} 
\caption{Special fibers of Mustafin varieties for $d = 3$
are degenerations of the projective plane.
The two schemes depicted above arise from
configurations of $n = 4$ points in $\mathfrak{B}_3$.
}
\label{fig:planar}
\end{figure}

This article is organized as follows. In Section 2 we 
develop the general theory of Mustafin varieties, including
their representation in terms of the polynomial ideals seen
in \cite{CS}.
Theorem \ref{sec:must;thm:basicfacts} summarizes the main geometric 
results, including the fact that the special fibers are
reduced, Cohen-Macaulay, and have rational components.

Section 3 concerns the case $d=2$, which was first studied by  Mumford in \cite[\S 2]{mum}.
Every configuration $\Gamma$ in the Bruhat-Tits tree
$\mathfrak{B}_2$ determines a finite phylogenetic tree $T_\Gamma$
which is an invariant of the isomorphism type of $\M(\Gamma)$.
In Theorem~\ref{thm:d-2-fiber} we determine the reduction 
complex of $\M(\Gamma)_k$ in terms of $T_\Gamma$, and in Proposition
\ref{prop:monotype}
we characterize configurations whose Mustafin
variety is defined by a monomial ideal.

The situation when $\Gamma$ lies in a single apartment  of $\mathfrak{B}_d$
is investigated in Section 4.
Theorem  \ref{thm:twistedVeronese} realizes
$\M(\Gamma)$ as a twisted Veronese variety.
The special fiber $\M(\Gamma)_k$ is the toric
degeneration of $\PP^{d-1}_k$ represented by
a regular mixed subdivisions of scaled simplices, as
seen on the left in Figure~\ref{fig:planar}.
The fact that any two points of $\mathfrak{B}_d$ lie in
one apartment leads to the classification of
Mustafin varieties for $n=2$ in Theorem~\ref{prop:classify2}.

In Section 5 we study the geometry of the irreducible
components of the special fiber $\M(\Gamma)_k$.
We distinguish between primary components,
which are indexed by  $\Gamma$ itself,
and secondary components, such as the 
bichromatic parallelograms in Figure \ref{fig:planar}.
Both types of components are rational but they can be
singular. Theorem \ref{thm:primary-components} characterizes
primary components as the blow-ups of projective spaces
along linear subspace arrangements.

Section 6 offers a detailed study of the case $n{=}d{=}3$,
centering around the algebro-geometric implications of the rich structure of
triangles in the two-dimensional building~$\mathfrak{B}_3$.
Our main result is the classification in Theorem~\ref{thm:TriangleCensus} of
Mustafin triangles into $38$ combinatorial types,
namely, the $18$ planar types in Figure \ref{fig:18types},
and $20$ non-planar~types.

\section{Structure of Mustafin Varieties}
\label{sec:struct-must-vari}

We denote by $\mathfrak{B}_d$ the Bruhat-Tits building associated to the group $PGL(V)$. 
It can be obtained by gluing certain real vector spaces, the \emph{apartments}.
Let $T$ be a maximal torus in $PGL(V)$. There is a basis $e_1, \ldots, e_d$ of $V$ such that $T$ is given by the group of diagonal matrices with respect to $e_1, \ldots, e_d$. The apartment in $\mathfrak{B}_d$ corresponding to $T$ is defined as  $A = X_\ast (T) \otimes_\Z \R$, where $X_\ast(T) = \operatorname{Hom} (\G_m, T)$ is the cocharacter group of $T$. 
We write $\eta_i$ for the cocharacter of $T$ induced by mapping $\lambda$ to the diagonal matrix with entry $\lambda$ in the $i$th place and entries $1$ in the other places.
The map $A \rightarrow \R^d / \R (1, \ldots, 1)$ that takes
  $\sum_i r_i \eta_i$ to the residue class of $(r_1, \ldots, r_d)$ is an isomorphism of vector spaces.

The apartment $A$ is the geometric realization of a simplicial complex
on the vertex set $X_\ast(T)\simeq \Z^d/\Z(1,\ldots, 1)$. This uses the isomorphism
above. Its simplices
are the cells in the infinite hyperplane arrangement that consists of
 the affine hyperplanes
\begin{equation}
\label{eq:hyperplanes}
\quad
H^{(ij)}_{m} \; = \; \biggl\{ \sum_{\ell=1}^d r_\ell \eta_\ell \in A: r_i - r_j = m \biggr\}
\qquad \hbox{
for $ 1 \leq i  < j \leq d$ and $m \in \Z$.}
\end{equation}
The building $\mathfrak{B}_d$ and its simplicial structure can be described in the following way. 
We write $[L] = \{ \alpha L: \alpha \in K^*\}$ for the \emph{homothety class} of a lattice $L$.  Two  lattice classes $[L']$ and $[M']$ are called \emph{adjacent} if there 
exist representatives $L$ and $M$ satisfying
\begin{equation}
\label{eq:adjacency}
\pi L \subset M \subset L.
\end{equation}
Let $[L]$ be a lattice class such that  $L$ is in diagonal form with respect to the basis $e_1, \ldots, e_d$,
i.e.\ $\,L = \pi^{m_1} R e_1 + \ldots + \pi^{m_d} R e_d$ for some $m_1, \ldots, m_d \in \Z$. We associate to $[L]$ the point $\sum_i (-m_i) \eta_i$ in the apartment $A$.
This is the standard bijection between the set of lattice classes in diagonal form with respect to $e_1, \ldots, e_d$ and the vertices of the  simplicial complex
above.  This bijection preserves adjacency. Hence the simplicial complex on~$A$ is the flag complex of the adjacency graph on diagonal lattice classes.

We write $\mathfrak{B}_d^0$ for the set of all lattice classes $[L]$  in~$V$.
Putting all apartments together, we see that the building $\mathfrak{B}_d$ is
a geometric realization of a simplicial complex on $\mathfrak{B}_d^0$, namely,
the flag complex of the graph on all lattice classes defined by the adjacency relation from~(\ref{eq:adjacency}).
The group $PGL(V)$ acts in the natural way on~$\mathfrak{B}_d$
and its vertex set~$\mathfrak{B}_d^0$.
If the residue field $k$ is a finite field containing $q$ elements, then $\mathfrak{B}_2$ 
is an infinite regular tree of valency $q+1$. More generally, the link of any vertex in
$\mathfrak{B}_d$ is isomorphic to the order complex of the poset of subspaces in $k^d$.
This follows from

\begin{lem}\label{sect:mus;lem:neighbours}
  Every neighbor of a vertex $[M]$ in $\mathfrak{B}_d^0$
 has the form $[L]$ for a lattice $L$ with $\pi M \subset L \subset M$, where  both inclusions are strict. Hence the quotient  $L / \pi M$ is a non-trivial subspace of the  $k$-vector space $M / \pi M$. In this way, we get a bijection 
\[\{\text{Neighbors of }[M]\} \longrightarrow \{\text{Non-trivial linear subspaces of }M / \pi M\},\]
mapping adjacent neighbors of $[M]$ to nested subspaces.
\end{lem}

Now let us choose coordinates and describe the
polynomial ideal that cuts out a Mustafin variety.
This ideal will be multihomogeneous in the sense of
the paper \cite{CS} whose notation and setup we shall adopt.
The image of the diagonal map $\Delta\colon \PP(V) \rightarrow \PP(V)^n = \PP(V) \times_K \ldots \times_K \PP(V)$ is the subvariety of the product $\PP(V)^n$ given by the  ideal $I_2(X)$ generated by the $2 \times 2$-minors of a matrix  $X = (x_{ij})_{\mathop{\scriptstyle i=1, \ldots, d}\limits_{\scriptstyle j = 1, \ldots, n}}$ of unknowns, where the $j$th column of this matrix represents coordinates on the $j$th factor.

Every $g \in GL(V)$ is represented by a matrix in $K^{d \times d}$, and it
 determines a dual (transpose) map $^t g\colon V^\ast \rightarrow V^\ast$
and a  morphism  $g\colon \PP(V) \rightarrow \PP(V)$.
This induces the usual action of $PGL(V)$ on $\PP(V)$. 
If $g_1, \ldots, g_n$ are elements of $GL(V)$,  the image of 
\[\PP(V)\, \stackrel{\Delta}{\longrightarrow} \,\PP(V)^n
\,\, \stackrel{g_1^{-1} \times \cdots \times g_n^{-1}}{\relbar\joinrel\relbar\joinrel\relbar\joinrel\relbar\joinrel\longrightarrow} \, \, \PP(V)^n \]
is the subvariety of the product $\PP(V)^n$ given by the multihomogeneous prime ideal
\begin{equation*}
I_2 \big((g_1, \ldots,g_n)(X)\big) \,\,\subset \,\,K[X].
\end{equation*}
Here $\,(g_1, \ldots, g_n) (X)\,$ is the 
$d {\times} n$-matrix whose $j$th column equals
\[g_j \left( \begin{array}{c} x_{1j} \\ \vdots \\ x_{dj} \end{array} \right).\]
Consider the reference lattice $L= Re_1 + \cdots + R e_d$. 
For any set of vertices $\Gamma = \{[L_1], \ldots, [L_n]\}$ in
the building $\mathfrak{B}_d$
we choose matrices $g_1,\ldots,g_n \in GL(V)$ such that $g_i L = L_i$
for all $i$.
The following diagram commutes:
\[
\xymatrix{ \PP(V) \, \ar[rr]^{(g_1^{-1}, \ldots, \,g_n^{-1}) \circ \Delta} \, \ar[d] & \,& \,\PP(V)^n \ar[d] \\
\prod_R \PP(L_i)\, \ar[rr]^{(g_1^{-1}, \ldots, \,g_n^{-1})}\, &\,& \,\PP(L)^n }
\]
Hence the Mustafin variety $\M(\Gamma)$ is isomorphic to the subscheme of 
$\,\PP(L)^n \simeq (\PP^{d-1}_R)^n\,$ 
cut out by the multihomogeneous ideal $\,I_2 \big((g_1, \ldots, g_n) (X)\big) \cap R[X]\,$
in $R[X]$.

\begin{example}[$d=n=3$] \label{ex:dreidrei}
 Let $K = \Q(\!(t)\!)$ and
$\Gamma$ the configuration determined by
\[g_1 = {\rm diag}(t^2,t,1)\,,\,\,
g_2 = {\rm diag}(t^4,t^2,1)\,,\,\,
g_3  = {\rm diag}(t^6,t^3,1).\]
Thus $\Gamma$ lies in the apartment specified by our choice of basis.
The generic fiber of the Mustafin variety $\M(\Gamma)$ is the subscheme
of $(\PP_K^2)^3$ defined by the $2 \times 2$-minors of
\begin{equation*}
(g_1,  g_2, g_3) (X)\quad = \quad
\begin{pmatrix}
x_{11} t^2  & x_{12} t^4  & x_{13}t^6 \,\\
x_{21}t       & x_{22} t^2 & x_{23}t^3 \\
x_{31}        & x_{32}   & x_{33}  
\end{pmatrix},
\end{equation*}
where $x_{1j}$, $x_{2j}$, and~$x_{3j}$ are the homogeneous coordinates of the $j$th factor of $(\PP_K^2)^3$.
The Mustafin variety $\M(\Gamma)$ itself is the intersection of this
ideal with the ring $R[x_{ij}]$.
The special fiber $\M(\Gamma)_\Q$ is the 
subscheme of $(\PP_\Q^2)^3$ defined by the monomial ideal
\begin{eqnarray*}
&
\bigl\langle
x_{11} x_{22},\,
x_{11} x_{32},\,
x_{21} x_{32},\,
x_{11} x_{23},\,
x_{11} x_{33},\,
x_{21} x_{33},\,
x_{12} x_{23},\,
x_{12} x_{33},\,
x_{22} x_{33} 
\bigr\rangle \\ =
& \,\,\,\,
\langle x_{11}, x_{21}, x_{12}, x_{22}  \rangle \,\cap\,
\langle x_{11}, x_{21}, x_{23}, x_{33} 	\rangle \,\cap\,
\langle x_{22}, x_{32}, x_{23}, x_{33} \rangle \,  \\ & \cap \,\,
\langle x_{11}, x_{21}, x_{12}, x_{33} \rangle \,\cap\,
\langle x_{11}, x_{32}, x_{12}, x_{33} \rangle \,\cap\,
\langle x_{11}, x_{32}, x_{33}, x_{23}\rangle.
\end{eqnarray*}
The first three components are isomorphic to $\PP_\Q^2$,
and the last three components are isomorphic to $\PP_\Q^1 \times \PP_\Q^1$.
This special fiber is the \emph{planar} monomial scheme in row 4 of \cite[Table 1]{CS}
and it is an instance of the tropical cyclic polytopes in \cite[\S 4]{BY}. 

By contrast, let us now consider the configuration $\Gamma'$ in $\mathfrak{B}_3$ determined by
\[g_1 = M_1 \cdot {\rm diag}(1,t,t^2) \,,\,\,
g_2 = M_2 \cdot {\rm diag}(1,t,t^2) \,,\,\,
g_3  = M_3 \cdot {\rm diag}(1,t,t^2) ,\]
where $M_1$, $M_2$ and~$M_3$ are generic $3 {\times} 3$-matrices over $\Q$.
Then $\M(\Gamma')$ is the subscheme
of $\,(\PP_R^2)^3\,$ obtained by saturation from the ideal of $2 \times 2$-minors of a matrix
\[ 
\begin{pmatrix}
\star x_{11} + \star x_{21} t + \star x_{31} t^2        &
\star x_{12} + \star x_{22} t + \star x_{32} t^2        &
\star x_{13} + \star x_{23} t + \star x_{33} t^2        \\
\star x_{11} + \star x_{21} t + \star x_{31} t^2        &
\star x_{12} + \star x_{22} t + \star x_{32} t^2        &
\star x_{13} + \star x_{23} t + \star x_{33} t^2        \\
\star x_{11} + \star x_{21} t + \star x_{31} t^2        &
\star x_{12} + \star x_{22} t + \star x_{32} t^2        &
\star x_{13} + \star x_{23} t + \star x_{33} t^2        
\end{pmatrix},
\]
where the stars indicate generic 
scalars in~$\Q$. The special fiber $\M(\Gamma')_\Q$ is
given by
\begin{eqnarray*}
&
\bigl\langle x_{11} x_{12}, x_{11} x_{22}, x_{21} x_{12},
x_{11} x_{13} , x_{11} x_{23} , x_{13} x_{21},
x_{12} x_{13} , x_{12} x_{23}, x_{13} x_{22} ,
x_{21} x_{22} x_{23} 
\bigr\rangle \\ =
& \,\,\,\,
 \langle x_{11}, x_{21}, x_{12}, x_{22} \rangle \,\cap\,
 \langle x_{12},x_{22},x_{13},x_{23} \rangle\,\cap\, 
\langle  x_{11},x_{21},x_{13},x_{23} \rangle \,
 \, \\ & \cap \,\,
 \langle x_{11}, x_{21}, x_{12}, x_{13} \rangle \,\cap\,
 \langle x_{11}, x_{12}, x_{22}, x_{13} \rangle \,\cap \,
 \langle x_{11}, x_{12}, x_{13}, x_{23} \rangle. 
 \end{eqnarray*}
This monomial scheme, denoted $Z$ in \cite[\S 2]{CS},
is the unique Borel-fixed point on the multigraded
Hilbert scheme $H_{3,3}$ of the diagonal embedding
$\,\PP^2 \hookrightarrow \PP^2 \times \PP^2 \times \PP^2$.
\qed
\end{example}

We have the following general structure theorem for Mustafin varieties.

\begin{thm}\label{sec:must;thm:basicfacts}
For a finite subset $\Gamma$ of $\mathfrak{B}_d^0$, the Mustafin variety $\M(\Gamma)$ is an integral, normal, Cohen-Macaulay scheme which is flat and projective over $R$.
Its generic fiber is isomorphic to the 
$(d{-}1)$-dimensional  projective space $\PP(V)$, and its special fiber
$\M(\Gamma)_k$ is reduced, Cohen-Macaulay and connected.
All irreducible components of $\M(\Gamma)_k$ are rational varieties,
and their number is at most
$\binom{n+d-2}{d-1}$, where $n = |\Gamma|$.
\end{thm}

\begin{proof}
By construction,
any Mustafin variety $\M (\Gamma)$ is irreducible, reduced and projective over $R$,
and with generic fiber~$\PP(V)$. Since $R$ is a discrete valuation ring, torsion-free implies flat, so $\M(\Gamma)$ is also flat over $R$. We show that the special fiber is connected by Zariski's Connectedness 
Principle~\cite[Theorem~5.3.15]{liu}. Since $\M(\Gamma)$ is proper over $R$, the
group of global sections $\sO_{\M(\Gamma)}(\M(\Gamma))$ is a finite $R$-module.
As it is contained in $\sO_{\PP(V)}(\PP(V))= K$, and $R$ is integrally closed,
we find indeed that the push-forward of $\sO_{\M(\Gamma)}$ is equal to $\sO_{\Spec R}$. Thus, the special fiber is connected.

Each Mustafin variety $\M(\Gamma)$ corresponds to an $R$-valued point in the
multigraded Hilbert scheme $H_{d,n}$ described in \cite{CS}, and its special
fiber $\M(\Gamma)_k$ is a $k$-valued point of $H_{d,n}$.  All $k$-valued points
of $H_{d,n}$ are reduced and Cohen-Macaulay by Theorem~2.1 and Corollary~2.6
in~\cite{CS}.  Since $\pi$ is a non-zero divisor on $\M(\Gamma)$ such that the
subscheme $\M(\Gamma)_k$ it defines is Cohen-Macaulay, $\M(\Gamma)$ is
Cohen-Macaulay along $\M(\Gamma)_k$. Away from $\M(\Gamma)_k$, 
the Mustafin variety $\M(\Gamma)$ is
regular, so $\M(\Gamma)$ is Cohen-Macaulay everywhere.  Finally, $\M(\Gamma)$ is
normal because it is is flat over a discrete valuation ring with normal generic
fiber and reduced special fiber~\cite[Lemma~4.1.18]{liu}.

The Chow ring of $(\PP^{d-1})^n$ (over any field) is $\mathcal A = \Z[H_1, \ldots, H_{n}]/\langle H_1^d, \ldots, H_{n}^d \rangle$, where $H_i$ represents the pullback of the hyperplane class from the $i$th factor. Up to change of coordinates, $\M(\Gamma)_K$ is embedded in $(\PP^{d-1}_K)^n$ as the diagonal. The codimension of this diagonal is $(d-1)(n-1)$ and its rational equivalence class is the sum over all monomials in~$\mathcal A$ of total degree $(d-1)(n-1)$ (see~\cite[Example~8.4.2\,(c)]{fulton} for the case $n=2$, which generalizes easily).
Since the special fiber~$\M(\Gamma)_k$ is a specialization of $\M(\Gamma)_K$, as in~\cite[Section~20.3]{fulton}, $\M(\Gamma)_k$ has the same class in~$\mathcal A$. This class is the sum of the classes of the components of~$\M(\Gamma)_k$. Since each component is effective, its class is 
a sum of non-negative multiples of monomials in~$\mathcal A$, and hence the number of
components is at most the number of  terms in the class of the diagonal, which is~$\binom{n+d-2}{d-1}$. 

The only remaining point is  that the components are
rational varieties. That proof
will be given in Section 5. 
The results in Sections 3 and 4 do not rely on~it.
\end{proof}

We note that the upper bound on the number of components is sharp. The class of examples realizing this upper bound is described below in Remark~\ref{rmk:monomial-max-comps}.

\vspace{1pc}

The following lemma enables us to take
closer look at the components of $\M(\Gamma)_k$.

\begin{lem}\label{lem:projection}
Let $\Gamma' \subset \Gamma$ be finite subsets of $\mathfrak B_d^0$. 
For each irreducible
component~$C$ of the special fiber~$\M(\Gamma')_k$, there is a unique
irreducible component of $\M(\Gamma)_k$ that maps
birationally onto~$C$ via the natural projection $\M(\Gamma)
\rightarrow \M(\Gamma')$.
\end{lem}

\begin{proof}
Let $\Gamma = \{[L_1], \ldots, [L_n]\}$ and $\Gamma' = \{[L_1], \ldots,
[L_{n'}]\}$ be a subset with $n' \leq n$. 
As above, let $\mathcal A = \Z[H_1, \ldots, H_{n}]/\langle H_1^d, \ldots, H_{n}^d \rangle$ be the Chow ring of $(\PP^{d-1})^n$. The class of $\M(\Gamma)_k$ is the sum of all monomials of total degree $(d-1)(n-1)$ and is equal to the sum of the classes of the components of~$\M(\Gamma)_k$. The class of each component is a sum
of non-negative multiples of monomials in~$\mathcal A$, and since the class of $\M(\Gamma)_k$ is multiplicity-free, each component must be a sum of distinct monomials in~$\mathcal A$. 
% The Chow ring of $(\PP^{d-1})^n$ (over any field) is $\mathcal A = \Z[H_1, \ldots, H_{n}]/\langle
%H_1^d, \ldots, H_{n}^d \rangle$, where $H_i$ represents the pullback of the
%hyperplane class from the $i$th factor. Up to change of coordinates,
%$\M(\Gamma)_K$ is embedded in $(\PP^{d-1}_K)^n$ as the diagonal. The
%codimension of this diagonal is $(d-1)(n-1)$ and its rational equivalence class
%is the sum over all monomials in~$\mathcal A$ of total degree $(d-1)(n-1)$ (see~\cite[Example~8.4.2\,(c)]{fulton} for the case $n=2$, which generalizes easily).
%Since the special fiber~$\M(\Gamma)_k$ is a specialization of $\M(\Gamma)_K$,
%as in~\cite[Section~20.3]{fulton}, $\M(\Gamma)_k$ has the same class
%in~$\mathcal A$. This class is the sum of the classes of the components
%of~$\M(\Gamma)_k$. Since each component is effective, its class must be a sum
%of non-negative multiples of monomials in~$\mathcal A$, and since the class of
%$\M(\Gamma)_k$ is multiplicity-free, each component must be a sum of distinct
%monomials in~$\mathcal A$.

Similarly, the class of the component $C$ of $\M(\Gamma')_k$ is the
sum of distinct monomials of degree $(d-1)(n'-1)$ in $\mathcal A' = \Z[H_1,
\ldots, H_{n'}]/\langle H_1^d, \ldots, H_{n'}^d\rangle$.
Let $H_1^{a_1} \cdots H_{n'}^{a_{n'}}$ with $a_1 + \cdots + a_{n'} =
(d-1)(n'-1)$ be one of them.  There is a unique component $\tilde C$ in
$\M(\Gamma)_k$ whose class contains the monomial $H_1^{a_1} \cdots
H_{n'}^{a_{n'}} \cdot H_{n'+1}^{d-1} \cdots H_{n}^{d-1}$.  Under the projection
$\M(\Gamma)_k \rightarrow \M(\Gamma')_k$, this class pushes forward to
$H_1^{a_1} \cdots H_{n'}^{a_{n'}}$.  Since $\M(\Gamma)_K \rightarrow
\M(\Gamma')_K$ is an isomorphism, under specialization, the rational equivalence
class of $\M(\Gamma)_k$ pushes forward to the class of $\M(\Gamma')_k$.
Thus, the projection of $\tilde C$ contains the monomial $H_1^{a_1} \cdots
H_{n'}^{a_{n'}}$. Since $C$ is the only component of $\M(\Gamma')_k$ containing
this monomial, it must be the image of $\tilde C$.  Furthermore, since the
coefficient of $H_1^{a_1}\cdots H_{n'}^{a_{n'}}$ is one, the map is birational
and $\tilde C$ is the unique component of $\M(\Gamma)_k$ with this property.
\end{proof}

\begin{cor}\label{cor:primary}
Let $\Gamma = \{[L_1], \ldots, [L_n]\}$ be a 
finite subset of $\mathfrak{B}_d^0$. For every  index $i$,
the special fiber $\M(\Gamma)_k$ has a 
 unique irreducible component $C_i$ with the property that $C_i$ maps
 birationally to $\PP(L_i)_k$ under the projection 
$\PP(L_1)_k \times \cdots \times \PP(L_n)_k \rightarrow \PP(L_i)_k$. 
\end{cor}

\begin{proof}
We take $\Gamma' = \{[L_i]\}$, so that $\M(\Gamma') = \PP(L_i)$, and apply
Lemma~\ref{lem:projection}.
\end{proof}

\begin{defn} An irreducible component of $\M(\Gamma)_k$ mapping birationally to the special fiber of the factor $\PP(L_i)$ for some $[L_i] \in \Gamma$ is called a \emph{primary component}. All  other components  of the special fiber are called \emph{secondary components}.
In both ideal decompositions of Example \ref{ex:dreidrei}, the first
three components are primary and the last three components are secondary.
For instance, the variety defined by $\langle x_{11}, x_{21}, x_{12}, x_{22}  \rangle$ is
$(0\mathop{:}0\mathop{:}1) \times (0\mathop{:}0\mathop{:}1) \times \PP^2_k$, and
this maps birationally (in fact, isomorphically)
onto the third factor of $\PP^2_k \times \PP^2_k \times \PP^2_k$.
\end{defn}

\begin{defn}
By an \emph{isomorphism of Mustafin varieties}
$\M(\Gamma)$ and~$\M(\Gamma')$
we mean an $R$-isomorphism between the schemes $\M(\Gamma)$ and~$\M(\Gamma')$ which
preserves the set of primary components. Thus, an isomorphism of Mustafin
varieties induces  a bijection between the defining lattice configurations $\Gamma$ and $\Gamma'$.
\end{defn}

We note that  two Mustafin varieties can be isomorphic as $R$-schemes
without being isomorphic as Mustafin varieties. This is shown in Example 
\ref{ex:HexagonWithThreeTriangles},
which exhibits a strict inclusion $\Gamma \subset \Gamma' $
such that the map $\M(\Gamma') \rightarrow \M(\Gamma)$ is an $R$-isomorphism.
The following result characterizes
the isomorphism classes of Mustafin varieties.

\begin{thm}\label{thm:same-config}
If $\M(\Gamma)$ and $\M(\Gamma')$ are isomorphic Mustafin varieties, then there
exists an element $g$ in $PGL(V)$ such that $\Gamma' = g \cdot \Gamma$ under the
action on subsets of $\mathfrak{B}_d$.
\end{thm}

\begin{proof}
If $\Gamma = \{[L_1], \ldots, [L_n]\}$ and $\Gamma' = g \Gamma$, then the isomorphism
\[(g, \ldots, g)\colon \mathbb{P}(L_1) \times_R \cdots \times_R \mathbb{P}(L_n) \longrightarrow \mathbb{P}(g L_1) \times_R \cdots \times_R \mathbb{P}(g L_n)\]
restricts to an isomorphism of Mustafin varieties $\M(\Gamma) \rightarrow \M(\Gamma')$, such that the induced map on the generic fiber $\mathbb{P}(V)$ is given by $g$.
Suppose conversely that $\varphi\colon \M(\Gamma) \rightarrow \M(\Gamma')$ is an isomorphism of Mustafin varieties. The  generic fiber of $\varphi^{-1}$ is given by an element $g \in PGL(V)$. As we have just seen, $g$ induces an isomorphsim of Mustafin varieties $\M(\Gamma') \rightarrow \M(g \Gamma')$. Hence after replacing $\Gamma'$ by $g \Gamma'$ we may assume that $\varphi$ is the identity map on the generic fiber. We claim that in this case $\Gamma = \Gamma'$.

Let $[L]$ be a lattice class in $\Gamma$, and let $C$ be the corresponding primary component of $\M(\Gamma)$. Since $\varphi$ is an isomorphism of Mustafin varieties, it maps $C$ to a primary component $C'$ of $\M(\Gamma')$, which corresponds to some lattice class $[L'] \in \Gamma'$.
We define the morphism $h\colon \M(\Gamma) \rightarrow \PP(L) \times_R \PP(L')$ as the product of the natural projection $\M(\Gamma) \rightarrow \PP(L)$ and the composition $\M(\Gamma) \stackrel{\varphi}{\rightarrow} \M(\Gamma') \rightarrow \PP(L')$, where $\M(\Gamma') \rightarrow \PP(L')$ is the natural projection. Then the generic fiber of $h$ is the diagonal embedding of $\PP(V)$ into $\PP(L) \times_R \PP(L')$. Therefore, $h$
induces a morphism from $\M(\Gamma)$ to the closure of $\PP(V)$ in $\PP(L)
\times_R \PP(L')$. Assuming that $[L]$ and $[L']$ are distinct lattice classes,
the closure is the Mustafin variety
$\M(\{[L], [L'] \})$. Note that  $h$ maps the primary component $C$ to the primary component $D$ of $\M(\{[L], [L']\})$ corresponding to $[L]$. The following diagram
is commutative and  $C$ maps birationally to $\PP(L')_k$:
\[
\xymatrix{ C \ar[r] \ar[d]  &  D \ar[d]  \\
\M(\Gamma)_k  \ar[r]^(.4){h} \ar[d]^{\varphi} &  \M(\{[L],[L']\})_k  \ar[d]\\
\M(\Gamma')_k \ar[r]  & \PP(L')_k }
\]
We conclude that the component $D$ is mapped birationally to
$\PP(L')_k$ under the projection on the right. However, by Corollary
\ref{cor:primary}, $D$ can't map birationally to both $\PP([L'])_k$
and~$\PP([L])_k$, so $[L]$ and~$[L']$ must be the same lattice point.
Hence, $\Gamma = \Gamma'$.
\end{proof}

With every Mustafin variety $\M(\Gamma)$ we associate a simplicial complex
representing the intersections between the irreducible components of its special fiber:

\begin{defn}
The \emph{reduction complex} of $\M(\Gamma)$ is the simplicial complex with one vertex for each component of the special fiber $\M(\Gamma)_k$, where a set of vertices forms a simplex if and only if the intersection of the corresponding components is non-empty. 
\end{defn}

Note that we also define reduction complexes in situations where the special fiber does not have simple normal crossings. In the case of normal crossings, our definition agrees with the standard one.

In Example \ref{ex:dreidrei}, the reduction complex of $\M(\Gamma)$ is a tetrahedron with triangles attached at two adjacent edges, while that of $\M(\Gamma')$ is the full $5$-simplex.
Reduction complexes for $d=2$ are characterized in Theorem~\ref{thm:d-2-fiber}.

We say that $\Gamma$ is \emph{convex} if whenever $[L]$ and~$[L']$ are in $
\Gamma$ then any vertex of the form $[\pi^a L \cap \pi^b L']$ is also in~$\Gamma$.
This is the notion of convexity used in \cite{fa} and in \cite{JSY}. The convex hull
of $\Gamma \subset \mathfrak{B}^0_d$ is the smallest convex subset of $\mathfrak{B}^0_d$ containing $\Gamma$.  We  call
$\Gamma$ \emph{metrically convex} if $\Gamma$ is closed under taking geodesics
in the natural graph metric on $\mathfrak{B}^0_d$, i.e.\ if $[L]$ and $[L']$
are in $\Gamma$ and ${\rm dist}([L],[L'']) + {\rm dist}([L''],[L']) = {\rm
dist}([L],[L'])$ then $[L'']$ is in $\Gamma$.  This equality 
holds for $L'' = \pi^a L \cap \pi^b L'$, so metrically convex implies 
convex, but not conversely.  Mustafin studied the varieties
$\M(\Gamma)$ only for metrically convex configurations $\Gamma$.
Note that in the context of Euclidean buildings there is yet another notion of convexity, which is
 induced from the Euclidean distances in apartments, but this notion of convexity does not play a role in our paper.

\begin{figure}
\centering
\begin{tabular}{ccc}
\includegraphics[width=5.1cm]{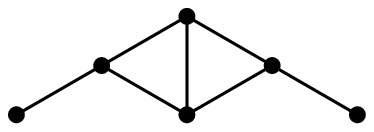} &\quad &
\includegraphics[width=2.4cm]{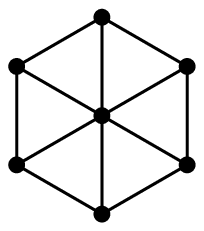} \\ &\quad & \\
\includegraphics[width=5.4cm]{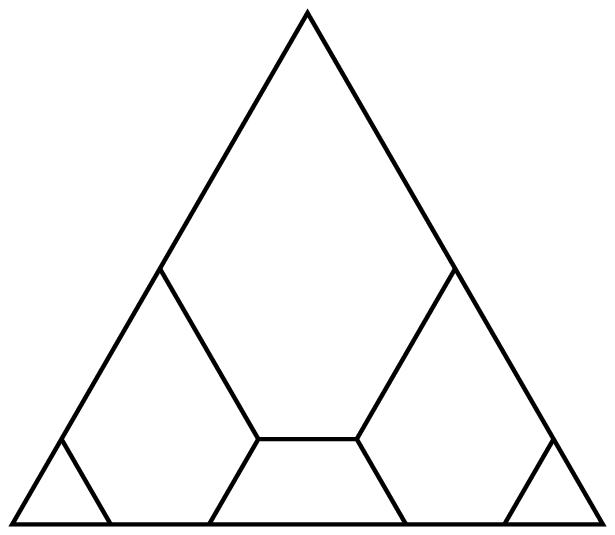}  & \quad &
\includegraphics[width=5.4cm]{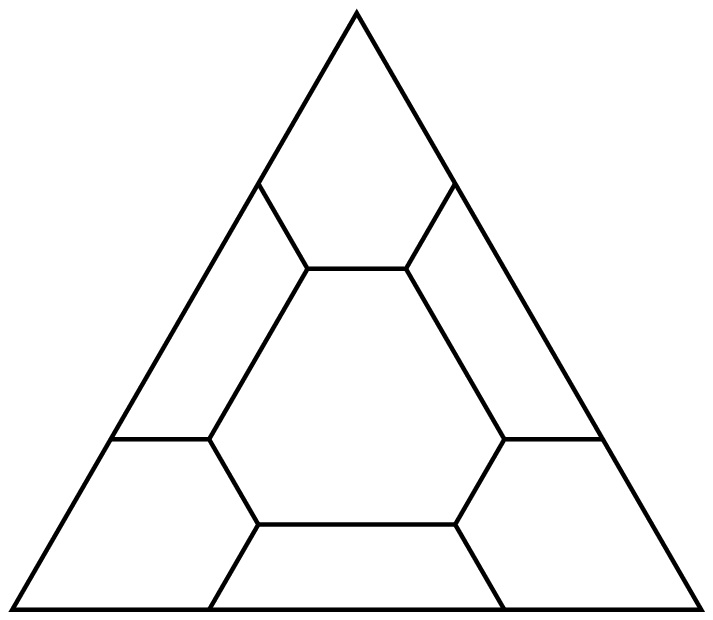}
\end{tabular}
\caption{Convex configurations in $\mathfrak{B}_3$ and the special fibers of their Mustafin varieties}
\label{fig:convex}
\end{figure}

The following theorem about convex configurations is illustrated by Figure  \ref{fig:convex}.

\begin{thm}\label{sec:must;thm:convex}
If $\Gamma$ is a convex subset consisting of $n$ lattice points in the building~$\mathfrak{B}_d$, then 
the Mustafin variety $\M(\Gamma)$ is regular, and its special fiber $\M(\Gamma)_k$ consists of $n$ smooth irreducible components that intersect transversely. In this case,
the reduction complex of $\M(\Gamma)$ is isomorphic to the simplicial subcomplex of $\mathfrak{B}_d$ induced by $\Gamma$. 
\end{thm}

\begin{proof}
Mustafin~\cite[Proposition~2.2]{mus} established this result for
configurations that are metrically convex, and we need to argue that it
also holds for all configurations that are convex in the sense above.
Under the convex hypothesis, Faltings~\cite{fa} showed that $\M(\Gamma)$ is regular and that
there are $n$ components in the special fiber that intersect transversely.
Note that Faltings uses the opposite convention for projective spaces, where points in $\PP(L)$ are hyperplanes (rather than lines) in $L$, so he also takes the dual notion of convexity, in terms of $L + L'$ instead of $L \cap L'$.
The smoothness of each irreducible component of $\M(\Gamma)_k$
follows from our Proposition~\ref{thm:nice-blowups} below.

We now prove the assertion about the reduction complex.
Consider the simplicial complex on $\Gamma$ induced from the
simplicial structure on $\mathfrak{B}_d^0$. The induced complex
is always a subcomplex of the reduction complex of $\M(\Gamma)$,
even if $\Gamma$ is not convex. Indeed, if $\overline \Gamma \supset \Gamma$
is the metric convex closure of $\Gamma$, then for any simplex in~$\Gamma$, the
corresponding components in $\M(\overline\Gamma)_k$ intersect by~\cite{mus}, and
hence, so do their images in $\M(\Gamma)_k$.

Suppose that $\Gamma$ is convex and
the  reduction complex  contains a simplex
that is not in the induced simplicial complex. Since
the latter is a flag complex, we can assume that the simplex is an edge
 $\{[L_1], [L_2]\}$. The two
 corresponding primary components intersect in
 $\M(\Gamma)_k$, and hence so do their
 components in $\M(\Gamma')_k$ where
 $\Gamma'  $ is the convex hull of $[L_1]$ and $[L_2]$ in $\mathfrak{B}_d$.
As in Proposition \ref{prop:classify2} below,
we can fix an apartment that contains both
$[L_1]$ and $[L_2]$. By construction, the
tropical line segment spanned by $[L_1]$ and $[L_2]$ in that apartment
has at least one additional lattice point $[L_3]$.
Consider the mixed subdivision $\Delta_{\{[L_1],[L_2],[L_3]\}}$
as in Section 4. A combinatorial argument
shows that the maximal cells indexed by
$[L_1]$ and $[L_2]$ do not intersect in that subdivision.
This contradicts to the assumption that
their primary components do intersect.
\end{proof}

For convex configurations $\Gamma$, the special fiber has only primary components,
and no secondary components. Without convexity assumptions,
a typical Mustafin variety has many secondary components.
Theorem~\ref{sec:must;thm:basicfacts} implies the following upper bound:
\begin{equation}
\label{eq:secondarynumber}
 \# \,\,\hbox{\rm secondary components of} \,\, \M(\Gamma)_k \,\,\, \leq \,\,\,
\binom{n+d-2}{d-1} - n.
\end{equation}
Note that for $d=2$, the special case of trees, the number above is zero.

\begin{remark}\label{rmk:monomial-max-comps}
The upper bound in
(\ref{eq:secondarynumber}) is attained 
when $\Gamma$ is of monomial type, as defined below.
This follows from the degree argument in the second-to-last paragraph of
the proof of Theorem \ref{sec:must;thm:basicfacts}
and the fact that the special fiber is always reduced.
\end{remark}

\begin{defn}
A configuration $\Gamma $ in the Bruhat-Tits building $\mathfrak{B}_d$ is 
of \emph{monomial type} if there exist bases for 
the $R$-modules $L_1, \ldots, L_n$ such that the multihomogeneous ideal 
in $k[X]$ that defines
$\M(\Gamma)_k$ is generated by monomials in the dual bases.
\end{defn}

We believe that monomial type is a generic condition.
To make this precise, we need to consider configurations of $n$ points with 
$\mathbb{Q}$-rational coordinates in $\mathfrak{B}_d$, and the statement would be that rational
configurations of monomial type are dense in the configuration space
for $n$ points in $\mathfrak{B}_d$. This would lead us define a
Gr\"obner fan structure on configuration spaces of buildings, a
topic we hope to return to in the future.

\section{Trees}

The building $\mathfrak{B}_2$ is an infinite tree. Any two points $v$ and $w$ in $\mathfrak{B}_2$ can be connected by a unique path. The lattice points in $\mathfrak{B}^0_2$
determine the simplicial structure on $\mathfrak{B}_2$.
We regard $\mathfrak{B}^0_2$ as a metric space, where 
adjacent lattice points have distance one.

Given any finite configuration $\Gamma \subset \mathfrak B^0_2$, the induced metric is a \emph{tree metric} on $\Gamma$. The tree that realizes this metric
is the  convex hull of $\Gamma$ in  $\mathfrak{B}_2^0$ with  the induced metric.
We denote this tree by $T_\Gamma$ and we 
refer to it as the  \emph{phylogenetic tree} of $\Gamma$.
Thus $T_\Gamma$ is a metric tree with $n$ labeled nodes
that include the leaves. Tree metrics
are studied in computational biology \cite[\S 2.4]{ascb}, where
it is well known that the phylogenetic tree
$T_\Gamma$  is uniquely determined by the metric on $\Gamma$.
The \emph{Neighbor-Joining Method}
\cite[Algorithm 2.41]{ascb} 
rapidly reconstructs the phylogenetic tree $T_\Gamma$
from the $\binom{n}{2}$ pairwise distances.

We are interested in the Mustafin variety $\M(\Gamma)$ specified by
the configuration $\Gamma \subset \mathfrak{B}_2^0$. First we show that
the metric tree $T_\Gamma$ can be read off the geometry of $\M(\Gamma)$. The following result is also proven in \cite[Proposition 2.3]{mum} by a different argument.

\begin{prop}\label{prop:d-2-comp}
If $\Gamma \subset \mathfrak{B}_2^0$, then each irreducible
component of the special fiber $\M(\Gamma)_k$  is
isomorphic to $\PP^1_k$, and
these irreducible components are in bijection with~$\Gamma$.
\end{prop}

\begin{proof}
By Theorem~\ref{sec:must;thm:basicfacts}, the special fiber
$\M(\Gamma)_k$ has at most $n$ components, where $n = |\Gamma|$.
By Corollary~\ref{cor:primary}, there are precisely $n$ primary components.
We conclude that every component is primary.  Also by Corollary~\ref{cor:primary},
each primary component $C$ maps birationally onto $\PP^1_k$. If $\tilde{C}$ denotes the
normalization of $C$, then the induced map $\tilde{C} \rightarrow C \rightarrow \PP^1_k$ must
be an isomorphism. Hence $C$ is isomorphic to $\PP^1_k$.
\end{proof}

\begin{figure}
\centering
\includegraphics[width=5.1cm]{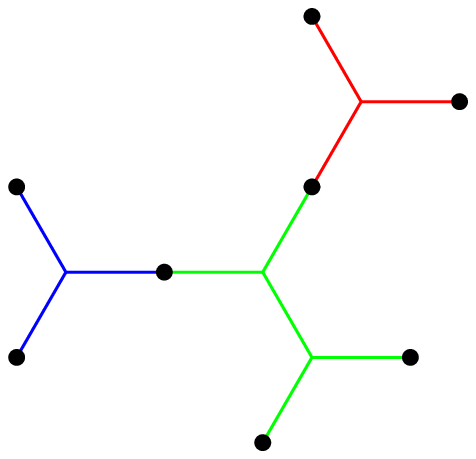} \qquad \qquad
\includegraphics[width=5.3cm]{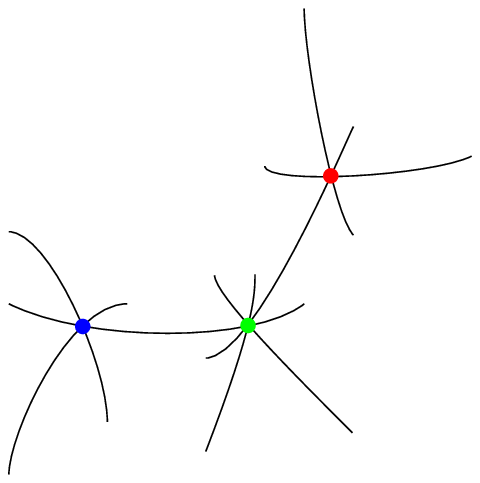}  
\caption{A configuration $\Gamma$ of $n=8$ points in 
$\mathfrak{B}_2$ whose associated phylogenetic tree $T_\Gamma$ has six leaves.
The corresponding special fiber is a tree of  eight
projective lines.}
\label{fig:tree}
\end{figure}

\begin{remark}
Because of Proposition~\ref{prop:d-2-comp}, in the rest of this section we will
speak interchangeably of the elements of~$\Gamma$ and the components
of the special fiber~$\M(\Gamma)_k$.
\end{remark}

 If $\Gamma$ consists of $n= 2$ points then their
distance $t$ 
has the following interpretation.

\begin{lem} \label{lem:tree2}
If $\Gamma= \{[L],[M]\}$, then the special fiber $\M(\Gamma)_k$ consists of two
projective lines $\,\PP_k^1\,$
that meet in one point. In a neighborhood of that point,
the Mustafin variety $\M(\Gamma)$ is defined
by a local equation of the form  $\,xy = \pi^t$, where $t = \mbox{dist}([L], [M])$.
\end{lem}

\begin{proof}
Since the two lattice classes
lie in a common apartment of the tree $\mathfrak{B}_2$,
there is a basis $\{e_1, e_2\}$ of~$V$ such that $L
= R e_1 + R e_2$ and $M = R e_1 + \pi^t R e_2$.
The Mustafin variety $\M(\Gamma)$ is  the subscheme of  $\,\PP^1_R \times \PP^1_R\,$
defined by the bihomogeneous  ideal $\,\langle x_2 y_1 - \pi^t x_1 y_2 \rangle$,
where $(x_1: x_2)$ and $(y_1 : y_2)$ are coordinates on the two factors.  
Hence the special fiber consists of two copies of
$\PP^1_k$ meeting transversely in one point.  In the affine coordinates $\,x= x_2/x_1\,$ and
$\,y = y_1/y_2$, the equation of $\M(\Gamma)$ becomes $\,x y = \pi^t$.
\end{proof}

The natural number $t$ is known as the \emph{thickness} of the singularity of
$\M(\Gamma)$.
The thickness~$t$ is invariant under changes of coordinates because in a minimal
resolution of singularities, there are exactly $t-1$ exceptional curves mapping
to the singular point on $\M(\Gamma)$. This is shown in~\cite[Lemma 10.3.21]{liu}.

We have
the following formula for the thickness $t$ in terms of two
matrices $g,h \in GL_2(K)$ that represent
  $L = g L_0$ and $M = h L_0$ relative to a reference lattice $L_0 \subset V$:
\begin{equation}
\label{eq:thickness}
 t \,\,\, = \begin{matrix}
v \bigl( {\rm det}(g_1 h_1) {\rm det}(g_2 h_2) - {\rm det}(g_1 h_2) {\rm det}(g_2 h_1)\bigr)
\qquad \qquad \qquad \qquad
\\ \quad - \,2 \cdot {\rm min}  \bigl\{
  v( {\rm det}(g_1 h_1)),  v( {\rm det}(g_1 h_2)),  v( {\rm det}(g_2 h_1)),  v( {\rm det}(g_2 h_2))
  \bigr\}. \end{matrix}
\end{equation}
Here $g_i$ and $h_j$ denote the columns of $g$ and $h$. To prove  (\ref{eq:thickness}), we note that 
$\M(\Gamma)_k$ is defined by
$\bigl\langle\,
 {\rm det}(g_1 h_1) x_{11} x_{12} + 
 {\rm det}(g_1 h_2) x_{11} x_{22} + 
  {\rm det}(g_2 h_1) x_{21} x_{12} + 
 {\rm det}(g_2 h_2) x_{21} x_{22} 
\bigr\rangle $,
and we change coordinates on $\PP^1_R \times \PP^1_R$
to eliminate the two middle terms.
The formula is invariant under coordinate transformations
that multiply $g$ or $h$ on the right by an element of $SL_2(R)$, but
we do not know a simple direct argument for this invariance.

The following theorem describes the correspondence between Mustafin varieties
$\M(\Gamma)$ and their phylogenetic trees $T_\Gamma$. 

\begin{thm}\label{thm:d-2-tree}
The isomorphism class of the Mustafin variety $\M(\Gamma)$ determines the tree $T_\Gamma$. 
Every phylogenetic tree  whose maximal valency is at most one more than the
cardinality of the residue field $k$
  arises in this manner from a configuration $\Gamma \subset \mathfrak{B}_2^0$.
\end{thm}

\begin{proof}
The first statement follows from Theorem \ref{thm:same-config}. Alternatively, we can argue using Lemma~\ref{lem:tree2}. 
Given the Mustafin variety $\M(\Gamma)$ as an $R$-scheme, the components
of its special fiber $\M(\Gamma)_k$ are labeled by $\Gamma$.
We then recover the tree metric $d_\Gamma$ on the convex hull $T_\Gamma$ as follows.
For $v,w \in \Gamma$, the projection
$\M(\Gamma)\rightarrow\M(\{v,w\})$  contracts all components of $\M(\Gamma)_k$ other than
$v$ and~$w$. 
By \cite[Proposition 8.3.28]{liu}, this contraction morphism between normal
fibered surfaces is
unique up to unique isomorphism. 
In a neighborhood of the intersection point of these two
components, $\M(\{v,w\})$ is defined by an equation of the form $xy- \pi^t$,
as in Lemma \ref{lem:tree2}.
The exponent $ t = d_\Gamma(v,w) $ is the distance
between $v$ and $w$ and coincides with the thickness of the singularity in $\M(\{v,w\})$.
Therefore we can construct the metric on $\Gamma$ from the geometry of $\M(\Gamma)$. 

Let $T$ be any phylogenetic tree with $n$ labeled leaves
and positive integral edge lengths. Assuming that its maximal valency is smaller or equal to
$|k| +1$,  we can embed $T$ isometrically into the building $\mathfrak{B}_2$ in such 
a way that the leaves are mapped to lattice points in $\mathfrak{B}_2^0$.
However, different embeddings may lead to non-isomorphic
Mustafin varieties. For example, if a vertex in $\Gamma$ has degree
four in $T_\Gamma$, then it intersects four other components, and the cross ratio
between the coordinates of these intersection points is an invariant of $\M(\Gamma)_k$.
Different cross ratios can occur for the same tree~$T_\Gamma$.
\end{proof}

We now discuss the special fiber~$\M(\Gamma)_k$, starting with its reduction complex.

\begin{thm}\label{thm:d-2-fiber}
The maximal simplices of the reduction complex 
of $\M(\Gamma)$ correspond to the connected
components of the punctured tree $T_\Gamma \backslash \Gamma$. The vertices in each maximal cell are the elements of $\Gamma$ in the closure of the corresponding component.
Thus, two irreducible components $v$ and $w$ of the special fiber $\M(\Gamma)_k$
intersect if
and only if the unique geodesic between $v$ and~$w$ in $T_\Gamma$ does not contain any other
vertex $u$ in~$\Gamma$. 
\end{thm}

\begin{proof}
Let $\,\overline \Gamma = T_\Gamma \cap \mathfrak{B}_2^0 \,$
be the set of all lattice points in the convex hull of  $\Gamma$.
Since $\Gamma \subset \overline \Gamma$,  we have a
projection from $\M(\overline\Gamma)$ to~$\M(\Gamma)$,
and hence from $\M(\overline\Gamma)_k$ to $\M(\Gamma)_k$.
By Theorem~\ref{sec:must;thm:convex}, two
components of the special fiber $\M(\overline \Gamma)_k$ intersect if and only if their vertices
are adjacent (i.e.\ have distance~$1$) in the simplicial structure on~$\mathfrak{B}_2$. 
Consider a  connected component $C$ of $T_\Gamma \backslash \Gamma$. If $C$ is an edge in $\mathfrak{B}_2$, then the adjacent vertices correspond to two intersecting components. Otherwise, the irreducible components of $\M(\Gamma)_k$ corresponding to the lattice points on $C$ form a $1$-dimensional connected subset of $\M(\Gamma)_k$.  Each of these components is contracted in the projection, so the
union of all components in $C$ projects to a single point in $\M(\Gamma)_k$. 

All irreducible components
in $\M(\overline\Gamma)_k$ corresponding to 
points of~$\Gamma$ lying in the closure of the connected component $C$ in $T_\Gamma$
intersect one of the components in $C$. Hence all these components 
 intersect in a common point in $\M(\Gamma)_k$,
and we conclude that the corresponding points of $\Gamma$ form a simplex
in the reduction complex.

It remains to be seen that there are no other simplices in the
reduction complex. In other words, we must show
that there are no other points of 
intersection between components of~$\M(\Gamma)_k$. Suppose that $[L_1]$
and~$[L_3]$ are not in the closure of a single connected component of
$T_\Gamma \backslash \Gamma$. Then the path between them contains at least
one other element $[L_2]$ of $\Gamma$. We explicitly compute the Mustafin
variety of the triple $\Gamma' = \{[L_1], [L_2], [L_3]\}$ in order to show that $[L_1]$ and~$[L_3]$
do not intersect in~$\M(\Gamma')_k$. Since all three points
lie in a common apartment of $\mathfrak{B}_2$,  
we can choose a basis $\{e_1,e_2\}$  for~$ V$
such that $L_i = R e_1 \oplus \pi^{s_i} R e_2$ with $0 = s_1 < s_2 < s_3$.
The ideal of $\M(\Gamma')$ is therefore
\begin{equation*}
  \bigl\langle \,\pi^{s_2}x_1y_2 - x_2y_1\,,\,
  \pi^{s_3}x_1y_3 - x_3y_1\,,\,
  \pi^{s_3-s_2}x_2y_3 - x_3y_2 \,\bigr\rangle,
\end{equation*}
where $(x_i:y_i)$ are the coordinates for
$\PP(L_i) \simeq \PP_R^1$. The ideal of the special fiber~is
\begin{equation*}
\langle x_2y_1, x_3y_1, x_3y_2\rangle \,\,=\,\,
\langle x_2, x_3\rangle
\cap \langle y_1, x_3 \rangle
\cap \langle y_1, y_2 \rangle.
\end{equation*}
The prime ideals on the right are the primary components for $[L_1]$, $[L_2]$ and~$[L_3]$.
 It is easy to see that the lines defined by the first and last
 ideal do not intersect in $(\PP_k^1)^3$.
In the projection from $\M(\Gamma)$ to $\M(\Gamma')$, the component of
$\M(\Gamma)_k$ indexed by $[L_i]$
maps to the corresponding component in $\M(\Gamma')_k$. The
projective lines indexed by $[L_1]$ and~$[L_3]$ are disjoint in~$\M(\Gamma)_k$ 
because their images are disjoint in $\M(\Gamma')_k$.
\end{proof}

\begin{ex} \rm
The reduction complex of the configuration in Figure \ref{fig:tree}
is a green tetrahedron which has two triangles, colored red and blue,
attached at two of its vertices. \qed
\end{ex}

\begin{remark}
Our discussion leads to the following description of  the singularities of $\M(\Gamma)_k$:
Every point of  the special fiber $\M(\Gamma)_k$ is locally isomorphic to a union of coordinate axes. 
This is shown in \cite[Proposition 2.3]{mum}. We can prove it as follows:
Recall that $\M(\Gamma)_k$ is a subscheme of $\prod\PP(L_i)_k$ and that each component
projects isomorphically to exactly one of the factors. Therefore, the inverse of
this isomorphism composed with one of the other projections must be constant. In
other words, the $i$th component is embedded as the product of $\PP(L_i)$ with a
point from each of the other factors $\PP(L_j)$ for $j \neq i$. Whenever
two or more components meet, they can be written as the union of coordinate axes.
The type of singularities appearing here is called \emph{TAC singularities} in \cite{to}, where degenerations of curves with this type of singularities are investigated.
\end{remark}

Our next goal is to explain the connection to the combinatorial 
data arising in the study of the multigraded Hilbert scheme $H_{2,n}$
in \cite[\S 4]{CS}. Among the $k$-points on that Hilbert scheme
are the special fibers of any Mustafin variety for $d=2$.
Let us now characterize the configurations of monomial type.  Recall that $\Gamma \subset \mathfrak{B}_2^0$ is 
of {monomial type} if  the ideal of
$\M(\Gamma)_k$ is generated by monomials in suitable bases.

\begin{prop} \label{prop:monotype}
A configuration $\Gamma \subset \mathfrak{B}^0_2$ has monomial type if and only if
every element of $\Gamma$ is either a leaf or is in the interior of an edge
in the phylogenetic tree~$T_\Gamma$.
\end{prop}

\begin{proof}
The configuration $\Gamma$ is of monomial type if and only if there is a linear torus
action on $\PP(L_1)_k \times \cdots \times \PP(L_n)_k$ such that
the special fiber $\M(\Gamma)_k$ is
invariant. On each factor $\PP(L_i)_k$, picking a linear torus action is equivalent to picking
two points~$0$ and $\infty$. Such a torus action can be lifted to an action on $\PP(L_1)_k \times \cdots \times \PP(L_n)_k$ leaving the other factors invariant. 
This action fixes the component of $\M(\Gamma)_k$
that maps isomorphically to $\PP(L_i)_k$.
   The other components however project to  points on $\PP(L_i)_k$. Hence a torus action on 
 $\PP(L_i)_k$ that leaves $\M(\Gamma)_k$ invariant exists if and only if
 there are at most two such points, and this happens if and only if
$[L_i]$ has degree at most~$2$ in~$T_\Gamma$.
\end{proof}

\begin{figure}
\centering
\vskip -0.2cm
\includegraphics[width=6.1cm]{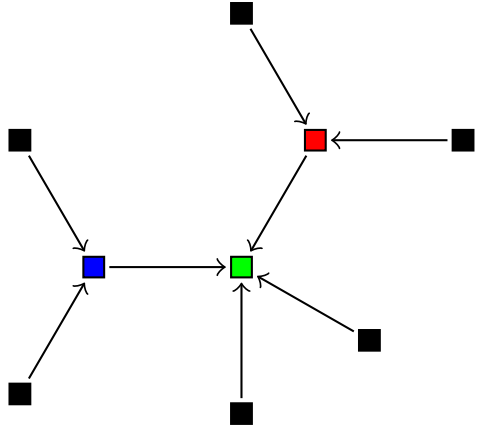} 
\caption{The Mustafin variety in Figure \ref{fig:tree} is of monomial type.
The monomial ideal of its special fiber is represented by
a tree with $9$ nodes and $8$ directed edges, as in \cite{CS}.}
\label{fig:monotree}
\end{figure}

According to \cite[Theorem 4.2]{CS},
 the Hilbert scheme $H_{2,n}$ has
precisely $2^n (n+1)^{n-2}$ points that represent
monomial ideals, and these are indexed by
 trees on $n+1$ unlabeled nodes with $n$ labeled directed edges.
For configurations $\Gamma$ of monomial type, we can  use Theorem~\ref{thm:d-2-fiber} to 
derive the \emph{monomial tree} representation of $\M(\Gamma)_k$
as in \cite[\S 4]{CS} from the phylogenetic tree
$T_\Gamma$. That combinatorial transformation of trees is as follows.

Assume that $n > 1$. 
We introduce one node for each leaf
of $T_\Gamma$ and one node for each connected component 
of $T_\Gamma \backslash \Gamma$. Thus the total number of
nodes is $n+1$.
For each $v \in \Gamma$ that is a leaf in $T_\Gamma$ we introduce one edge
between the node
corresponding to $v$  and the component of $T_\Gamma \backslash \Gamma$
that is incident to $v$.
Each non-leaf $v \in \Gamma$ is incident to two components
of $T_\Gamma \backslash \Gamma$, by Proposition \ref{prop:monotype}, and
we introduce an edge between these nodes as well. Thus the total
number of edges is $n$. Now, the monomial tree is directed by
orienting each of the $n$ edges 
according to the choice of $0$ and~$\infty$
for that copy of $\PP^1$. In particular, 
if every point of $\Gamma$ is a leaf
in $T_\Gamma$, then the monomial tree is the star tree $Z$
in \cite[Example 4.9]{CS}.
By reversing this construction, we can show that each of the
$2^n (n+1)^{n-2}$  monomial trees arises from a configuration
in the Bruhat-Tits tree $\mathfrak{B}_2$.

\begin{cor}
Every monomial ideal in the multigraded Hilbert scheme $H_{2,n}$
arises as the special fiber of a Mustafin variety $\M(\Gamma)$
for some $n$-element set $\,\Gamma \subset \mathfrak{B}_2^0$.
\end{cor}

\section{Configurations in One Apartment}
\label{sec:conf-one-apartm}

We will now investigate Mustafin varieties $\M(\Gamma)$ determined by
configurations $\Gamma$ that are contained in a single
apartment $A$ of the Bruhat-Tits building  
$\mathfrak{B}_d$.
The apartment  equals $A = X_\ast (T) \otimes_\Z \R$ for
a maximal torus $T$, and we identify
\[ A  \,\, = \,\, \R^d / \R(1, 1,\ldots, 1). \]
Scholars in tropical geometry use the term
 {\em tropical projective torus} for the apartment $A$ together with its integral
structure $A \cap \mathfrak{B}_d^0$.
We recall from Section~\ref{sec:struct-must-vari} that 
\begin{equation}
\label{bijection} \{\pi^{m_1} R e_1 + \ldots + \pi^{m_d} R e_d \} \mapsto (-m_1, \ldots, -m_d) + \R(1,\ldots, 1)
\end{equation}
is a bijection between the set of lattice classes in diagonal form with respect
to the basis $e_1,\ldots, e_d$ and the set of vertices in $A$. 
We define the distance between two points $u = (u_1,\ldots, u_d)$  and $v =
(v_1,\ldots, v_d) $ in the apartment $ A$ as the variation
\[ \dist (u,v) =  \max_i \{u_i - v_i\} - \min_i  \{u_i - v_i\} \,\,=\,\,
\max \bigl\{ u_i -v_i - u_j + v_j \,: \, i \neq j   \bigr\}. \]
If $u$ and $v$ are lattice points in~$A$, then $\,\dist (u,v)\,$ coincides with the
combinatorial distance between lattice classes which we discussed in
Section 2.
 
The apartment $A$ is a \emph{tropical semimodule}
 with tropical vector addition
 $\,\min\{u, v\} \,=\, (\min\{u_1,v_1\},\ldots, \min\{u_d,v_d\})\,$ 
and   scalar multiplication
$\,\lambda + \, u \,=\, (\lambda + u_1,\ldots,\lambda + u_d)$.
 A subset $S$ of $A$ is \emph{tropically convex}
 if, for all $u,v \in S$ and $\lambda, \mu \in \R$, the
 element $\min \{\lambda + u, \mu + v\}$ is
  also in $S$. For relevant basics on tropical convexity see \cite{ds, JSY}.

  We fix a finite configuration $\Gamma = \{u^{(1)},\ldots,u^{(n)}\}$ 
  of lattice points in $A \cap \mathfrak{B}_d^0$. The point 
  $u^{(i)} = (u_{i1}, \ldots,u_{id})$ represents the diagonal lattice
  $\,L^{(i)} = \pi^{-u_{i1}} R e_1 + \cdots +  \pi^{-u_{id}} R e_d$.
  Tropical convex combinations correspond to 
  convex combinations of diagonal lattices in~$\mathfrak B_d$. 
  This is made precise by the following lemma.
    Recall (e.g.~from \cite{BY, ds, djs, JSY}) that the \emph{tropical polytope} or \emph{tropical convex hull}, $  \operatorname{tconv}(\Gamma)$,
  is the smallest tropically convex subset of $A$ containing~$\Gamma$.
  The following lemma shows that this tropical convex hull corresponds to the
  convex hull of a set of lattice classes defined in Section~\ref{sec:struct-must-vari}. 
  
  \begin{lem} The bijection (\ref{bijection}) induces a bijection between
   the lattice points in $\operatorname{tconv}(\Gamma)$
and the lattice classes in the convex hull of 
$\Gamma$ in the Bruhat-Tits building~$\mathfrak{B}_d$.
  \end{lem}

\begin{proof}
Suppose $L$ and $L'$ are lattices that lie in the apartment $A$ and let
$u$ and $u'$ be the vectors in $\Z^d$ that represent them.
The lattice $\pi^s L \cap \pi^t L'$ is also in $A$, and is represented by
$\min\{s + u , t+ u'\}$. Therefore, the two notions of convex combinations coincide.
\end{proof}

We now consider the dual tropical structure on $A$ given by
the max-plus algebra. We encode $\Gamma$ by the corresponding
product of linear forms in the max-plus algebra:
\begin{equation}
\label{eq:prodlinear}
 P_\Gamma(X_1, \ldots, X_d) \; = \;
\sum_{i=1}^n
\max (-u_{i1}  + X_1, -u_{i2}+ X_2, \ldots , -u_{id} + X_d) .
\end{equation}

The tropical hypersurface $\mathcal{T}(P_\Gamma)$ defined by this expression
is an arrangement of $n$
tropical hyperplanes (see~\cite{ad}) in the $(d-1)$-dimensional space $A$. Dual to this
tropical hyperplane arrangement is a \emph{mixed subdivision}~\cite[\S 1.2]{san} of the scaled 
standard simplex
\begin{equation}
\label{eq:SumOfSimplices}
 \,n \Delta_{d-1} \,\,=\,\, \Delta_{d-1} + \Delta_{d-1} + \cdots + \Delta_{d-1}.
 \end{equation}
We denote this mixed subdivision by $\Delta_\Gamma$.
Each cell of $\Delta_\Gamma$ has the form
\begin{equation}
\label{eq:cellDecomp}
 \sigma \,\, = \,\, F_1 + F_2 + \cdots + F_n , 
 \end{equation}
where $F_i$ is a face of the $i$th summand $\Delta_{d-1}$
in the Minkowski sum (\ref{eq:SumOfSimplices}).

The combinatorial relationship between the mixed subdivision, the
tropical hyperplane arrangement and the tropical polytope
$\operatorname{tconv}(\Gamma)$ were introduced
 in~\cite[\S 5]{ds} and further developed in~\cite{ad, djs}.
The cells of the mixed subdivision $\Delta_\Gamma$
are in order-reversing bijection with the cells in the
  tropical hyperplane arrangement
 determined by $\mathcal{T}(P_\Gamma)$.
The tropical polytope $\operatorname{tconv}(\Gamma)$ is the union of the
bounded cells in the arrangement $\mathcal{T}(P_\Gamma)$. These
bounded cells correspond to  the interior cells of $\Delta_\Gamma$.

\begin{ex} \rm
For $d=3$ there are many pictures of the above objects in the literature.
For instance, for $n=3$ consider the configuration $\Gamma$ 
from Example \ref{ex:dreidrei} which is represented by the points
$u^{(1)} = (-2,-1,0)$,
$u^{(2)} = (-4,-2,0)$ and
$u^{(3)} = (-6,-3,0)$ in $A = \R^3/\R(1,1,1)$.
That type of tropical triangle is obtained by moving the black point
in \cite[Figure 5]{djs} towards the southwestern direction.
All $35$ combinatorial types for $n=4$ are depicted in 
\cite[Figure 6]{ds}. 
Note that the mixed subdivision $\Delta_\Gamma$ shown on the left in our Figure \ref{fig:planar}
corresponds to the configuration $\Gamma$ that is labeled  \texttt{T[34]} in the census of \cite{ds}.
Our Figure~\ref{fig:convex} shows configurations $\Gamma$ that consist of
\emph{all} lattice points in a tropical polygon.
Two cells in their mixed subdivisions $\Delta_\Gamma$
intersect if and only if the corresponding points in $\Gamma$
are connected by an edge in the simplicial complex structure on 
$ \operatorname{tconv}(\Gamma)$.
This is the statement about the reduction complex in Theorem \ref{sec:must;thm:convex}. \qed
\end{ex}

Since the planar case ($d=3$) has been amply visualized, we chose
a configuration of three points in the three-dimensional apartment
to serve as our example with picture.

\begin{figure}
\centering \vskip -0.2cm
\includegraphics[width=.36\textwidth]{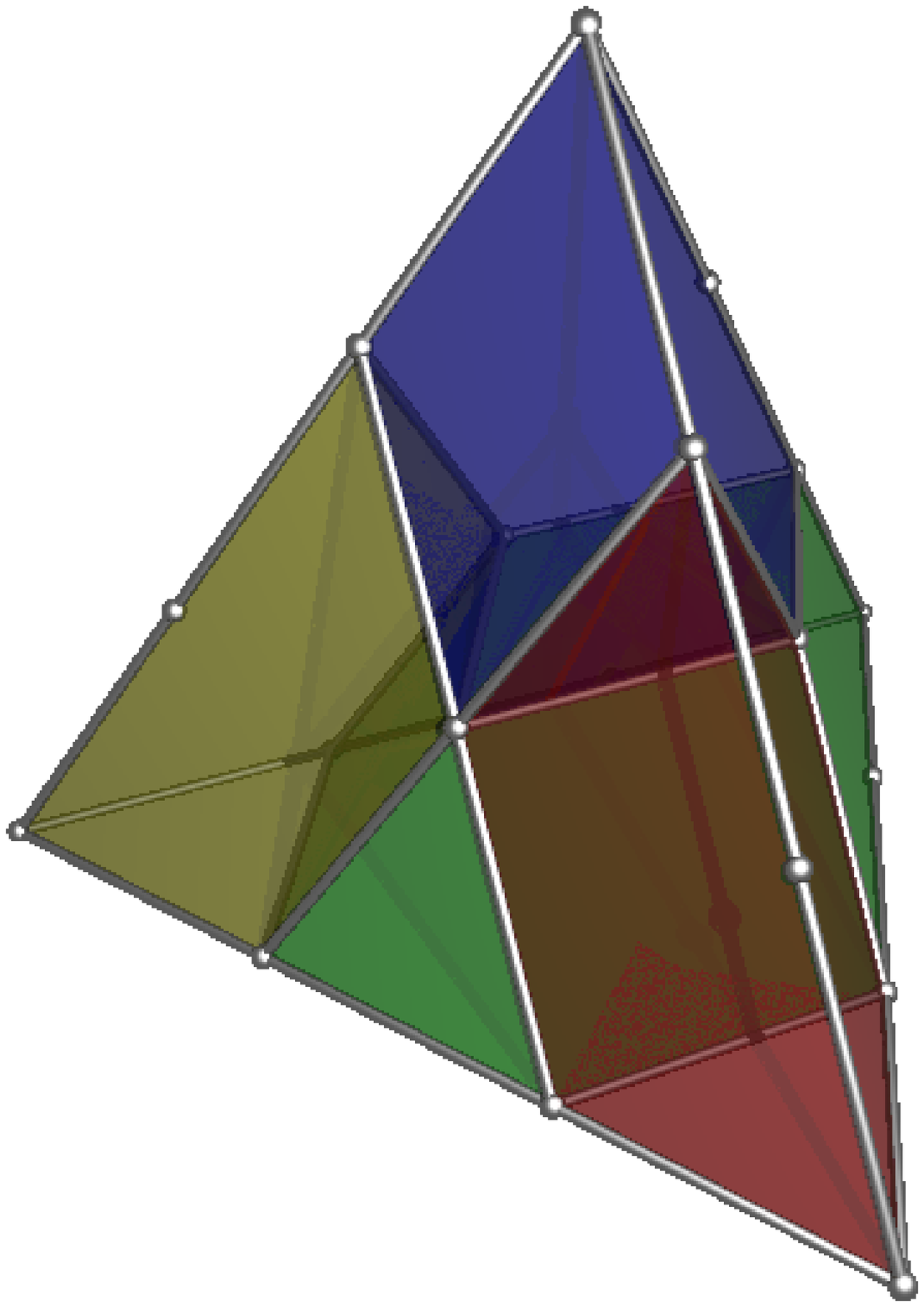} \hspace{0.4in}
\includegraphics[width=.44\textwidth]{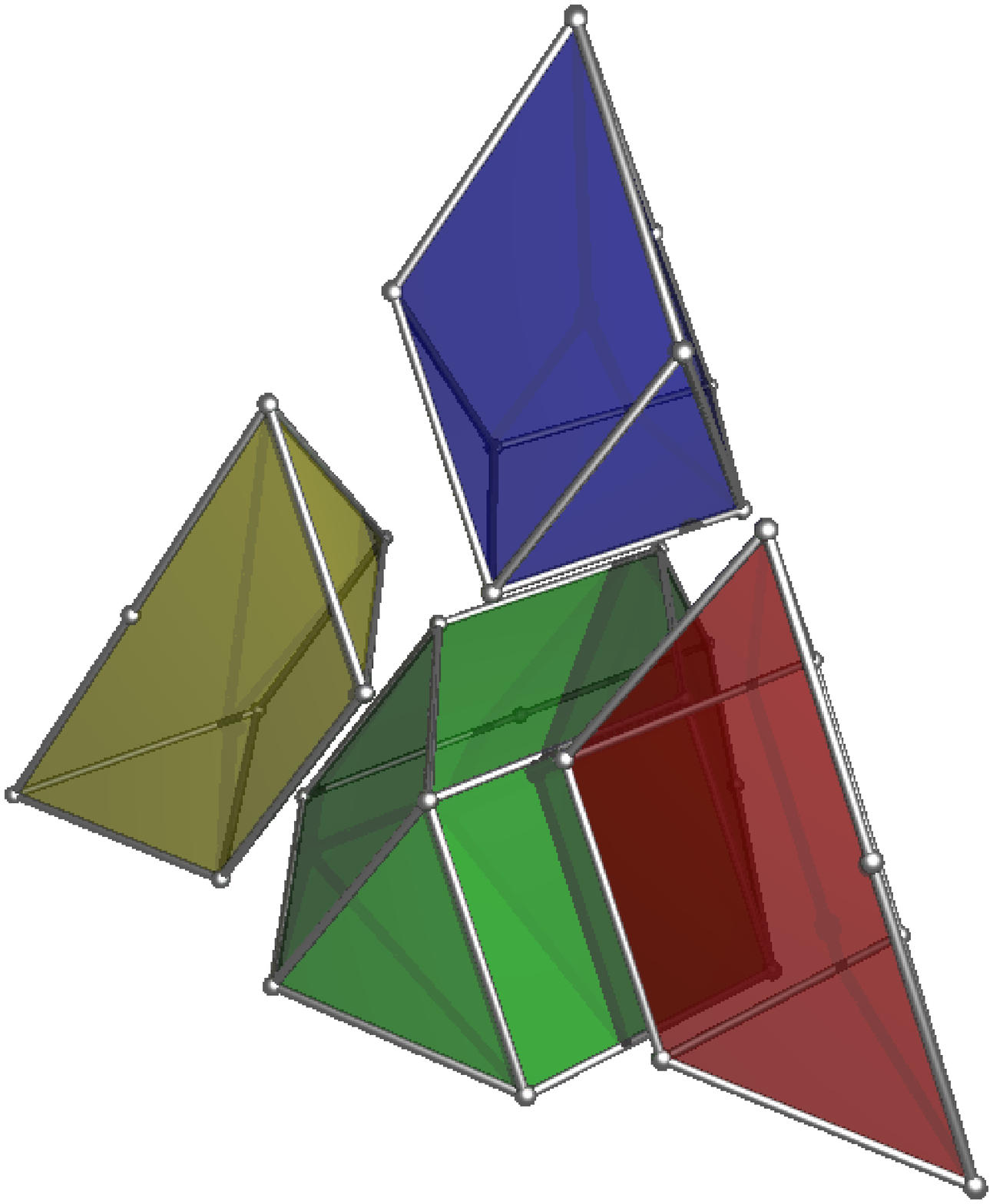} \hspace{0.1in}
\caption{The special fiber of the Mustafin variety $\M(\Gamma)$
for the configuration $\Gamma$  in Example \ref{ex:explodingtetrahedron}
is represented by a mixed subdivision of the tetrahedron into four cells.
The green cell represents a
primary component of $\M(\Gamma)_k$ that is a singular toric $3$-fold.}
\label{fig:exploding}
\end{figure}

\begin{ex} \label{ex:explodingtetrahedron} \rm
Fix $d = 4$, $n = 3$ and  the vertices
$u^{(1)} {=} (0,0,0,0)$, 
$u^{(2)} {=} (0,-1,0,-1)$, 
$u^{(3)} {=} (0,0,-1,-1)$.
This configuration
$\Gamma$ defines an arrangement of three tropical planes in $A = \R^4/\R(1,1,1,1)$ 
with defining equation
\[
P_\Gamma \,=\, 
\max(X_1,X_2,X_3,X_4) + 
\max(X_1,X_2+1, X_3,X_4+1) + \max(X_1,X_2,X_3+1,X_4+1). 
\]
The expansion of this tropical product of linear forms is
 the tropical cubic polynomial
\begin{equation}
\label{exploding1}
 \begin{matrix}
  P_\Gamma \,=\, \max \{\,
 3 X_1+c_{111},
 2 X_1 {+} X_2 + c_{112},  \ldots,
 X_1 {+} X_2 {+} X_3 + c_{123}, \ldots, 3 X_4 + c_{444} \}, \\ \qquad \text{with} \quad
                              c_{111} = 0\,,
                              \,\,\,
c_{112} = c_{113} = c_{114} = c_{122} = c_{133} = c_{222} = c_{333} = 1, 
 \\ \qquad \qquad \quad
  c_{123} {=} c_{124} {=}  c_{134} {=} c_{144} {=} c_{223} {=} c_{224}
{=} c_{233} {=} c_{234} {=} c_{244} {=} c_{334} {=} c_{344} {=} c_{444} = 2.
\end{matrix} \end{equation}
 We associate these $20$ coefficients with
 the lattice points in the  tetrahedron $3 \Delta_3$ of
(\ref{eq:SumOfSimplices}). The corresponding regular mixed
 subdivision $\Delta_\Gamma$  of the tetrahedron induced by lifting
 the points to heights $c_{ijk}$. It has four maximal cells 
(\ref{eq:cellDecomp}) and is shown in Figure \ref{fig:exploding}. \qed
 \end{ex}

These combinatorial objects are important for the study of
Mustafin varieties because of the following theorem, which is the main result in
this section.

\begin{thm} \label{thm:twistedVeronese}
The Mustafin variety $\M(\Gamma)$ is isomorphic to the twisted $n$th Veronese embedding
of the projective space $\PP_R^{d-1}$ determined by the
tropical polynomial $P_\Gamma$. In particular,
the special fiber $\M(\Gamma)_k$ equals the union
of projective toric varieties corresponding to the cells in the
 regular mixed subdivision $\Delta_\Gamma$ of 
the simplex $n \cdot \Delta_{d-1}$.
\end{thm}

We now give a precise definition of the \emph{twisted Veronese embedding}
referred to above.
Fix the lattice $L = R\{e_1, \ldots, e_d\}$ and corresponding
projective space $\PP(L)$ with coordinates $x_1, \ldots, x_d$.
 The $n$th symmetric power 
$\operatorname{Sym}\nolimits_n(L)$ of $L$ is a free $R$-module of rank $\binom{n+d-1}{n}$.
The corresponding projective space $\PP(\operatorname{Sym}\nolimits_n(L))$ has coordinates
$y_{i_1 i_2 \cdots i_n}$ where $1 \leq i_1 \leq i_2 \leq \cdots \leq i_n \leq d $.
We can embed $\PP(L)$ into $\PP({\rm Sym}_n(L))$
by the $n$th \emph{Veronese embedding} which is given in coordinates by
$\,y_{i_1 i_2 \cdots i_n} = x_{i_1} x_{i_2} \cdots x_{i_n}$.

Consider any homogeneous tropical polynomial of degree $n$,
\begin{equation}
\label{anyoldC}
  C \,\,=\,\,
\max \bigl\{ -c_{i_1 i_2 \cdots i_n} + X_{i_1} + X_{i_2} + \cdots + X_{i_n} \,:
1 \leq i_1 \leq i_2 \leq \cdots \leq i_n \leq d \bigr\}, 
\end{equation}
where the coefficients $c_{i_1 i_2 \cdots i_n}$
are arbitrary integers.
We define the \emph{twisted Veronese embedding}
$\PP(V) \rightarrow_C \PP(\operatorname{Sym}\nolimits_n(L)) $
corresponding to the tropical polynomial $C$ by
\[  y_{i_1 i_2 \cdots i_n} \,\,=\,\, x_{i_1} x_{i_2} \cdots x_{i_n} \cdot
\pi^{c_{i_1 i_2 \cdots i_n}}
\qquad \hbox{for} \quad
1 \leq i_1 \leq i_2 \leq \cdots \leq i_n \leq d. \]
The closure of the image of this morphism in
$\PP(\operatorname{Sym}\nolimits_n(L))$
is an irreducible $R$-scheme. Its generic fiber 
is isomorphic to the $n$th Veronese embedding of $\PP^{d-1}_K$.

\begin{proof}[Proof of Theorem~\ref{thm:twistedVeronese}]
Our twisted Veronese embedding is a special case of the 
general construction of toric degenerations of projective toric
varieties. That construction is well-known to experts, and it
is available in the literature at various levels of generality,
starting with the work on Gr\"obner bases of toric varieties,
and the identification of initial monomial ideals with regular
triangulations, presented in \cite{stu}. A complete treatment for
the case of arbitrary polyhedral subdivisions, but still over
$\C$, appears in \cite[\S 4]{hu}.
The best reference for our setting
of an arbitrary discretely valued field $K$ seems to be  Alexeev's construction of
one-parameter families of stable toric pairs in \cite[\S 2.8]{al}

The special fiber of the twisted Veronese embedding
is a scheme over $k$. According to  \cite[Lemma 2.8.4]{al},  
its irreducible components are the projective toric varieties
corresponding to the polytopes in the regular polyhedral 
subdivision of $n \cdot \Delta_{d-1}$ induced by the heights $c_{i_1 i_2 \cdots i_n}$.
These toric varieties are glued according to the dual cell structure given by the
tropical hypersurface $\mathcal{T}(C)$. In the tropical literature this construction
is known as {\em patchworking}. We note that the special
fiber of the twisted Veronese can be non-reduced, even for
$n=d=2$, as with  $\langle  y_{12}^2 - \pi y_{11} y_{22} \rangle $.

In this proof we do not consider arbitrary tropical polynomials
 but only  products of linear forms. Those encode configurations
   $\Gamma \subset A \cap \mathfrak{B}_d^0$. Their
coefficients $\, c_{i_1 i_2 \cdots i_d} \,$ are so special that they
ensure  the reducedness of the special fiber.
Equating the tropical polynomial $C$ in~\eqref{anyoldC} with the
tropical product of linear forms $P_\Gamma$ 
in (\ref{eq:prodlinear}), we obtain
\[ c_{i_1 i_2 \cdots i_n} \,\, = \,\,
\min \bigl\{
u_{i_1 \sigma_1} + 
u_{i_2 \sigma_2} + \cdots + 
u_{i_n \sigma_n} \,: \, \sigma \in \mathfrak{S}_n \bigr\}\text{.} \]
Here $\mathfrak{S}_n$ denotes the symmetric
group on $\{1,2,\ldots,n\}$.
The regular polyhedral subdivisions of $n \cdot \Delta_{d-1}$
defined by such choices of coefficients are the mixed subdivisions.

Consider the sequence of two embeddings
\[ \PP(V) \hookrightarrow  \PP(L_1) \times_R \cdots 
\times_R \PP(L_n) \hookrightarrow
\PP( L_1 \otimes_R \cdots \otimes_R L_n) \text{.}\]
The map on the left is the one in the definition of
the Mustafin variety $\M(\Gamma)$. The map on 
the right is the classical \emph{Segre embedding}, which is given in coordinates by
\[ z_{i_1 i_2 \cdots i_n}  =  x_{i_1, 1} \, x_{i_2, 2}\, \cdots\, x_{i_n,n} \qquad
\hbox{for $\,\,1 \leq i_1,i_2, \ldots, i_n \leq d$}. \]
The Segre variety is the toric variety in
$\PP( L_1  \otimes_R \cdots \otimes_R  L_n) \,$ cut out by the equations
\[ \quad
z_{i_1 \cdots i_n} \cdot z_{j_1 \cdots j_n} =
z_{k_1 \cdots k_n} \cdot z_{l_1 \cdots l_n} 
\quad \text{ whenever } \{i_\nu,j_\nu\} = \{k_\nu,l_\nu\} 
\,\, \text{ for } \nu = 1,2,\ldots,n.
\]
The image of $\PP(V)$ in the Segre variety is cut out by the linear equations
\[ z_{j_1 j_2 \cdots j_n} \,
\pi^{u_{1,k_1} + \cdots + u_{n, k_n}}  = 
z_{k_1 k_2 \cdots k_n} \,
\pi^{u_{1,j_1} + \cdots + u_{n, j_n}}
\]
whenever the multisets $\{j_1,j_2,\ldots,j_n\}$ and
 $\{k_1,k_2,\ldots,k_n\}$ are equal.
For any ordered sequence of indices $1 \leq i_1 \leq \cdots \leq i_n \leq d$, we
introduce the coordinates
\[
y_{i_1 i_2 \cdots i_n} \; := \;
z_{j_1 j_2 \cdots j_n} \, \pi^{c_{j_1 j_2 \cdots j_n}} \]
where $(j_1, \ldots,j_n)$ is a permutation of
$(i_1,\ldots,i_n)$ such that $\,c_{j_1 \cdots j_n} = u_{1,i_1}  + \cdots + u_{n,i_n}$.
Substituting the $y$-coordinates for the $z$-coordinates in the above
equations, we find that the image of $\PP(V)$ in
$\PP(L_1 \otimes_R \cdots \otimes_R L_n)$ equals the twisted Veronese variety.
\end{proof}

\begin{ex} \label{ex:HexagonWithThreeTriangles} \rm
Let $d=n=3$ and $\Gamma = \{(1,0,0), (0,1,0), (0,0,1)\}$.
The tropical convex hull of $\Gamma$ is $\overline{\Gamma} = \Gamma \cup \{(0,0,0)\}$,
as seen in \cite[Figure 4]{JSY}.
The map $\M(\overline\Gamma) \rightarrow \M(\Gamma)$
is an isomorphism of schemes over $R$ but it is not an isomorphism
of Mustafin varieties. The special fiber $\M(\Gamma)_k$
has four irreducible components. It consists of the
blow-up of $\PP_k^2$ at three points with a copy of $\PP_k^2$ glued 
along each exceptional divisor. The central component is
primary in $\M(\overline\Gamma) $ but it is demoted to being secondary
in $ \M(\Gamma)$. \qed
\end{ex}

The tropical polytope $\operatorname{tconv}(\Gamma)$ 
comes with two natural subdivisions into 
classical convex polytopes. First, there is the
\emph{tropical complex} which is dual to the
mixed subdivision  $\Delta_\Gamma$. Second,
there is the simplicial complex induced from $A$ on
$\operatorname{tconv}(\Gamma)$. This simplicial complex
refines the tropical complex, and it is generally much finer.
The vertices of the tropical complex
on~$\operatorname{tconv}(\Gamma)$ correspond to the
facets of~$\Delta_\Gamma$, and hence to the irreducible
components of the special fiber $\M(\Gamma)_k$.
As before, we distinguish between primary 
and secondary components, so the
tropical complex on $\operatorname{tconv}(\Gamma)$ has
both primary and secondary vertices.
The primary vertices are those contained in~$\Gamma$, and
the secondary vertices are all other vertices
of the tropical complex. 

The points in $\Gamma$ are in {\it general position}
if every maximal minor of the $d \times n$-matrix $(u_{ij})$ is tropically
non-singular.
In this case the  number of vertices in the tropical complex is $\binom{n+d-2}{d-1}$.
That number is ten for $d=3, n=4$, as seen in \cite[Figure~5]{djs},
and the ten vertices correspond to the ten polygons in pictures as the
left one in Figure~\ref{fig:planar}. The points in $\Gamma$ are
in general position if and only if 
the subdivision  $\Delta_\Gamma$  of $n \Delta_{d-1}$ is a \emph{fine} mixed subdivision,
arising from a  \emph{triangulation} of $\Delta_{d-1} \times \Delta_{n-1}$~\cite[Prop.~24]{ds}.

\begin{prop}
For a set $\Gamma$ of $n$ elements in $A \cap \mathfrak{B}_d^0$ the
following are equivalent:
\begin{itemize}
\item[(a)] The configuration  $\Gamma$ is in general position.
\item[(b)] The special fiber $\M(\Gamma)_k$ is of monomial type.
\item[(c)] $\M(\Gamma)_k$ is defined
 by a monomial ideal in $k[X]$ in our chosen coordinates. 
\item[(d)] The number of secondary components of $\M(\Gamma)_k$ equals
$\,\binom{n+d-2}{d-1}-n$.
\end{itemize}
\end{prop}

\begin{proof}
Clearly, (c) implies (b), and Remark \ref{rmk:monomial-max-comps} states that (b) implies (d).
By~\cite{san}, the mixed subdivision of $n\Delta_{d-1}$ is equivalent to a
subdivision of $\Delta_{d-1} {\times} \Delta_{n-1}$.
That polytope is unimodular and has
normalized volume $\binom{n+d-2}{d-1}$. Hence a 
polyhedral subdivision of $\Delta_{d-1} {\times} \Delta_{n-1}$ is a triangulation
if and only if it has  $\binom{n+d-2}{d-1}$ maximal cells, so
 (a) is equivalent to (d). The equivalence of (a) and (c) is proven
in \cite[Prop.~4]{BY}.
\end{proof}

Block and Yu \cite{BY} showed that the tropical complex on
$\operatorname{tconv}(\Gamma)$ can be computed as a minimal free
resolution in the sense of commutative algebra when
$\Gamma$ is in general position. Namely, they show
that the Alexander dual to the monomial ideal in (b) has
a unique minimal cellular resolution, and the support
of that resolution is the tropical complex on $\operatorname{tconv}(\Gamma)$.
This construction was extended to non-general tropical point configurations
$\Gamma$ by Dochtermann, Joswig and Sanyal \cite{djs}.

Our results in this section apply to arbitrary Mustafin varieties when $n=2$.
Indeed, let $L^{(1)}$ and $L^{(2)}$ be any two lattices in $V$ and consider 
$\Gamma = \{ [L^{(1)}], [L^{(2)}]\}$.
Then there exists an apartment $A$ that contains $\Gamma$, 
and we can represent both $L^{(i)}$ by vectors $u^{(i)} \in \Z^d$
as above. The configuration $\{ u^{(1)}, u^{(2)} \}$ is in general position
if and only if the quantities
$u_{1i} + u_{2j} - u_{2i} - u_{1j}$ are non-zero for all $1 \leq i < j \leq d$. This is equivalent
to the statement that $u^{(1)}$ and $u^{(2)}$ are not contained in any affine hyperplane of the form
(\ref{eq:hyperplanes}) in $A$. 
Assuming that this is the case, the tropical complex on
the line segment $\operatorname{tconv}(\Gamma)$ consists of $d-1$
edges and $d-2$ secondary vertices between them.
Each of the $d-1$ edges is further subdivided into 
segments of unit length in the simplicial complex structure.
If $\Gamma$ is not in general position then
some edges may have length zero.

\begin{prop} \label{prop:classify2}
For $n=2$, isomorphism classes of  Mustafin varieties are in bijection with
lists of $d-1$ non-negative integers, up to reversing the order. The elements of
the list are the lengths of the segments of the one-dimensional tropical complex
 $\operatorname{tconv}(\Gamma)$.
\end{prop}

\begin{proof}
We apply Theorem~\ref{thm:same-config}. Since any two points lie in a
single apartment, their convex hull consists of a tropical line segment. The
lengths along these line segments are an invariant of the configuration.
\end{proof}

A coarser notion of isomorphism is given by the
{\em combinatorial type} of the mixed subdivision $\Delta_\Gamma$.
Here, two configurations $\Gamma$ and $\Gamma'$ in
$A \cap \mathfrak{B}_d^0$ have the same combinatorial type
if and only if the special fibers $\M(\Gamma)_k$ and $\M(\Gamma')_k$
are isomorphic as $k$-schemes.

For fixed $d$ and $n$, there are finitely many combinatorial types
of configurations~$\Gamma$. As mixed subdivisions of $n \Delta_{d-1}$
are in bijection with the triangulations of $\Delta_{n-1} {\times} \Delta_{d-1}$,
the combinatorial types are classified by the faces of the
{\em secondary polytope} $\Sigma(\Delta_{d-1} {\times} \Delta_{n-1})$.
We illustrate this classification for the case $d=n=3$ of triangles
in the tropical plane.

\begin{example}[$n=d=3$]\label{ex:planar-3-3}
The {\em secondary polytope}
 of the direct product of two triangles,
 $\Sigma(\Delta_2 \times \Delta_2)$, is a four-dimensional
polytope with f-vector $(108, 222, 144, 30)$.
The $108$ vertices correspond to the $108$ triangulations
of $\Delta_2 {\times} \Delta_2$. These come in five combinatorial types,
first determined by A.~Postnikov, and later displayed in \cite[Figure 39, p.~250]{GKZ}.
The special fibers corresponding to these five types are listed in rows
1-5 of \cite[Table 1]{ds}, and they are depicted in the first row of
Figure \ref{fig:18types} below. The second picture shows an orbit of size $12$,
and it represents a type which is incorrectly drawn in 
\cite[Figure 39]{GKZ}: the rightmost vertical edge
should be moved to the left.

The $222$ edges of  $\Sigma(\Delta_2 {\times} \Delta_2)$  come 
in seven types, shown in the second and third row of Figure~\ref{fig:18types}.
The $144$ $2$-faces come in five types, shown in the last two rows
of Figure~\ref{fig:18types}. In each case we report also the number of
``bent lines,'' by which we mean pairs of points in $\Gamma$
whose tropical line segment disagrees with the classical line segment.
Finally, the $30$ facets of $\Sigma(\Delta_2 {\times} \Delta_2)$  
come in three types, corresponding to the coarsest mixed subdivisions of $3 \Delta_2$.
Two of them do not appear for us
because they are degenerate in the sense that
two of the three points in $\Gamma$ are the identical.
The only facet type that corresponds to a valid Mustafin triangle $\Gamma$
is shown on the lower right in Figure~\ref{fig:18types}. \qed
\end{example}

\begin{figure}
\vskip -0.65cm
 \begin{enumerate}
 \item {\em $6$ components, $3$ bent lines, corresponding to vertices of the
 secondary polytope:}
 \vskip .02cm
    \begin{tabular}{rccccccccc}
     $\fbox{108} = \!\!\!\!\! $ & $\!\! 6$ & $\!\! + \!\!$ & $12$ & $\!\!+\!\!$ & $18$ & $\!\!+\!\!$ & $36$ & $\!\!+\!\!$ & $36$ \\   \!\!\!  & \!\!\!\! \includegraphics[scale=.56]{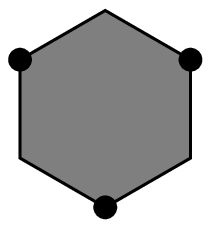} \!\!\!&& \!\!\! \includegraphics[scale=.5]{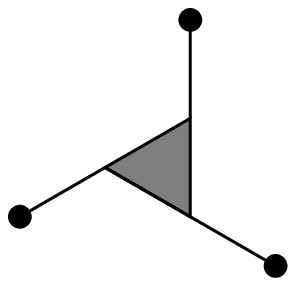} && \!\!\!\! \ \includegraphics[scale=.5]{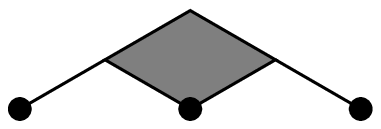} \!\!\! && 
\! \!\!\! \includegraphics[scale=.5]{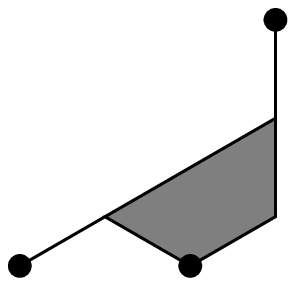}\!\!\!\! && \!\!\!\!\includegraphics[scale=.5]{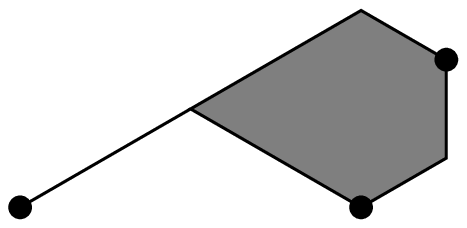} \!\!\! \\ \!\!\!
     & \!\!\!\! \includegraphics[scale=.5]{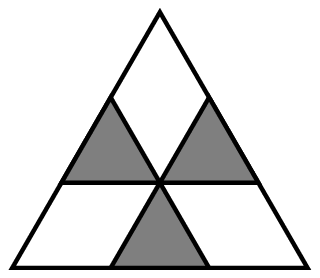} \!\!\! && \!\!\!\includegraphics[scale=.5]{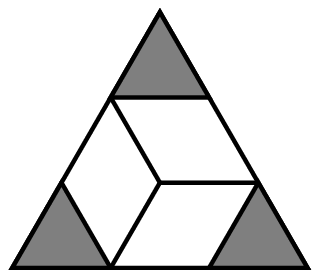} \!\!\! && \!\!\! \includegraphics[scale=.5]{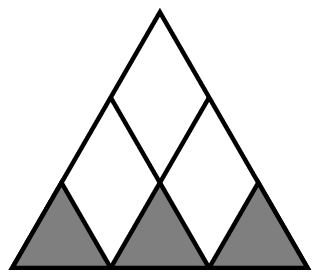} \!\!\! && \!\!\! \includegraphics[scale=.5]{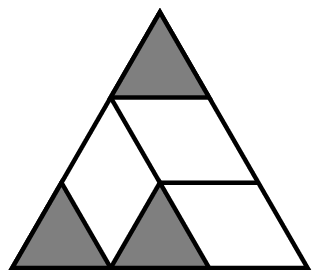}\!\!\! &&\!\!\! \includegraphics[scale=.5]{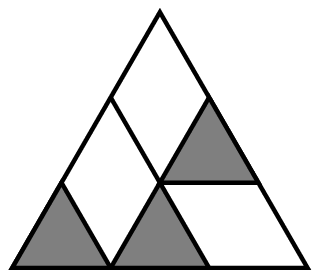} 
   \end{tabular}
 \item {\em $5$ components, $2$ bent lines, corresponding to edges of the secondary polytope:}
 \vskip .02cm
   \begin{tabular}{rccccccccc}
     $\fbox{180} = \!\!\!\!\! $ & $\!36$ & $\!\!+\!\!$ & $36$ & $\!\!+\!\!$ & $36$ & $\!\!+\!\!$ & $36$ & $\!\!+\!\!$ & $36$ \\
\!\!\!     &\!\!\!\! \includegraphics[scale=.5]{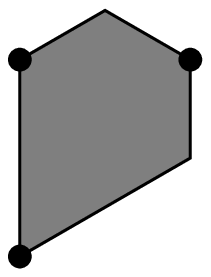} \!\!\! && \!\!\!\includegraphics[scale=.5]{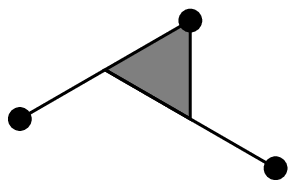} \!\!\! && \!\!\! \includegraphics[scale=.5]{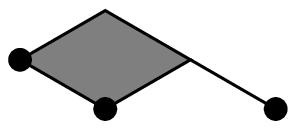} \!\!\! && \!\!\! \includegraphics[scale=.5]{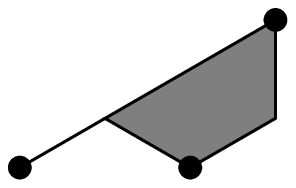}\!\!\! && \!\!\! \includegraphics[scale=.5]{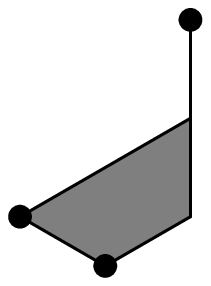} \\
     \!\!\! &\!\!\! \includegraphics[scale=.5]{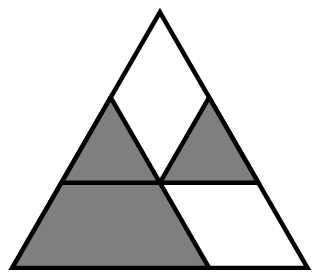} &&\!\!\! \includegraphics[scale=.5]{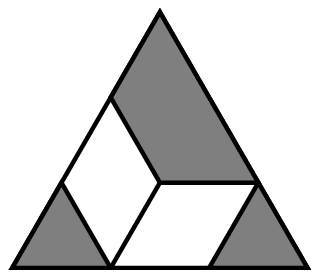} \!\!\! && \!\!\! \includegraphics[scale=.5]{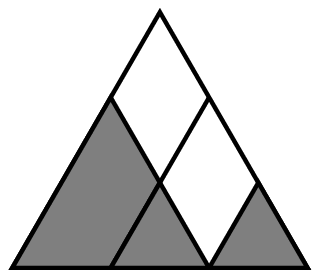}  && \includegraphics[scale=.5]{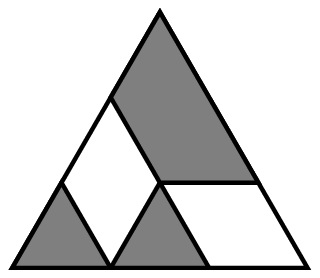} &&     \!\!\! \includegraphics[scale=.5]{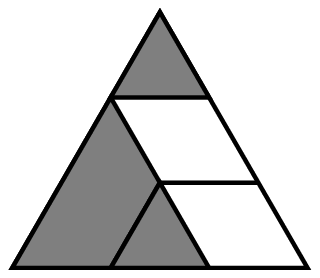} 
   \end{tabular}
 \item {\em $4$ components, $3$ or $2$ bent lines, corresponding to edges of the secondary
 polytope:}
 \vskip .02cm
   \begin{tabular}{rccrc}
     $\fbox{6}$ & (3 bends) & \qquad & $\fbox{36}$ & (2 bends)\\
     & \includegraphics[scale=.5]{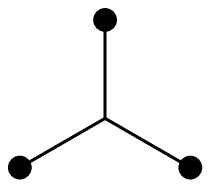} & \qquad && \includegraphics[scale=.45]{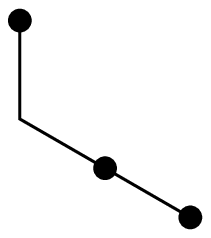} \\
     & \includegraphics[scale=.42]{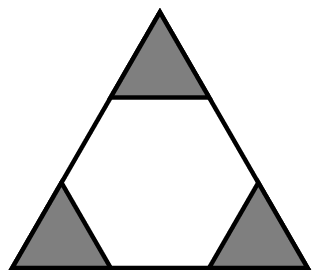} & \qquad && \includegraphics[scale=.42]{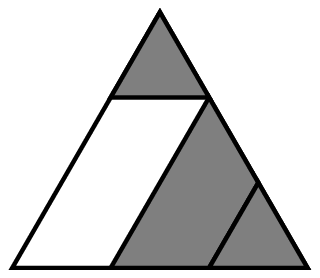} \\
   \end{tabular}
 \item {\em $4$ components, $1$ bent line, corresponding to $2$-faces of
 the secondary polytope:}
 \vskip .02cm
    \begin{tabular}{rccccc}
     $\fbox{90} = $ & $18$ & $+$ & $36$ & $+$ & $36$ \\
     & \includegraphics[scale=.6]{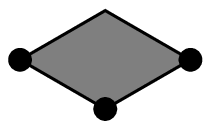} && \includegraphics[scale=.6]{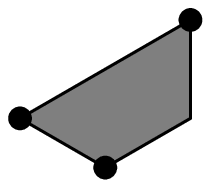} && \includegraphics[scale=.6]{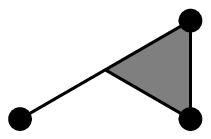}  \\
     & \includegraphics[scale=.43]{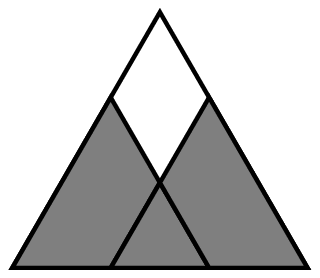} && \includegraphics[scale=.43]{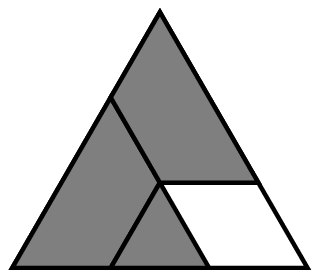} && \includegraphics[scale=.43]{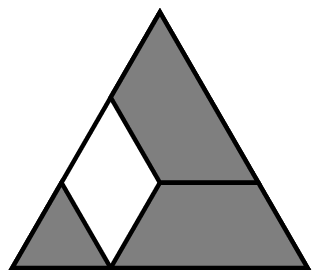} 
   \end{tabular}
 \item  {\em $3$ components, $1$ or $0$ bent lines, corresponding to $2$-faces or facets:}
 \vskip .02cm
   \begin{tabular}{rccrccrc}
     $\fbox{18}$ & (1 bend, 2-face) & \qquad & $\fbox{36}$ & (0 bends, 2-face) & \qquad & $\fbox{12}$ & (0 bends, facet)\\
     & \includegraphics[scale=.67]{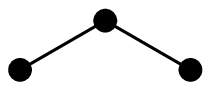} & \qquad && \includegraphics[scale=.67]{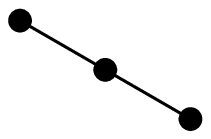} & \qquad && \includegraphics[scale=.67]{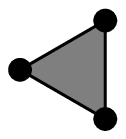} \\
     & \includegraphics[scale=.45]{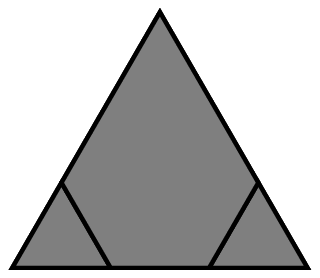} & \qquad && \includegraphics[scale=.45]{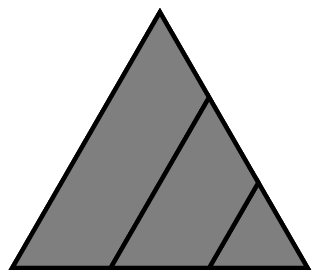} & \qquad && \includegraphics[scale=.45]{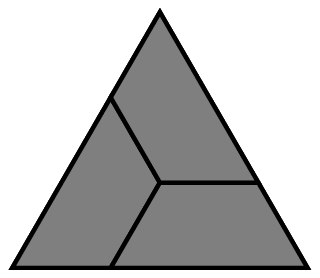} \\
   \end{tabular}
 \end{enumerate}
 \vskip -0.5cm
 \caption{The $18$ combinatorial types of planar Mustafin triangles}
 \label{fig:18types}
\end{figure}

\section{Components of the Special Fiber}

In this section, we describe the components of~$\M(\Gamma)_k$. If $d=2$
then each component is isomorphic to~$\PP_k^1$, by Proposition \ref{prop:d-2-comp}.
If $\Gamma$ is contained in one apartment then each component is
a toric variety, by  Theorem \ref{thm:twistedVeronese}.
 In general, we shall obtain the primary components of $\M(\Gamma)_k$
   from projective space $\PP_k^{d-1}$ by blowing up linear subspaces.

\begin{defn} \label{def:blowup}
Let $W_1, \ldots, W_m$ be linear subspaces in $\PP^{d-1}_k$.
Let $X_0 = \PP^{d-1}_k$ and inductively define $X_i$ to be the
blowup of $X_{i-1}$ at the preimage of $W_i$ under $X_{i-1} \rightarrow \PP^{d-1}_k$.
We say that $X$ is a \emph{blow-up of $\PP^{d-1}_k$ at a collection of linear subspaces} if $X$ is isomorphic to
the variety~$X_m$ obtained by this sequence of blow-ups.
\end{defn}

We now show that the order of the blow-ups does not matter in
Definition \ref{def:blowup}:

\begin{lem}\label{lem:blow-up-universal}
Let $\pi \colon X \rightarrow \PP^{d-1}_k$ be the blow-up at the
linear spaces $\{W_1,\ldots,W_m\}$.
  Then $X$ is isomorphic to the blow-up of 
$\mathscr{I}_1 \mathscr{I}_2 \cdots \mathscr{I}_m$
where $\mathscr{I}_i$ the ideal sheaf of~$W_i$.
\end{lem}

\begin{proof}
Let $\pi_i \colon X_i \rightarrow \PP^{d-1}_k$ be the projection in the $i$th
stage of Definition~\ref{def:blowup} and let $\rho \colon Y \rightarrow
\PP^{d-1}_k$ be the blow-up at the product 
ideal sheaf $\mathscr{I}_1 \cdots
\mathscr{I}_m$. By the universal property of blowing up,
$\pi_i^{-1}(\mathscr{I}_i)\cdot \mathscr{O}_{X_i}$ is locally free. Thus, locally,
this sheaf is generated by one non-zero
section of $\mathscr{O}_{X_i}$. Since $X \rightarrow X_i$ is surjective and $X$ is
integral, the pullback of this section is a non-zero section of $\mathscr{O}_X$, so
$\pi^{-1}(\mathscr{I}_i) \cdot \mathscr{O}_{X}$ is also locally free. Therefore,
the product
$\pi^{-1}(\mathscr{I}_1) \cdots \pi^{-1}(\mathscr{I}_m)\cdot \mathscr{O}_X = 
\pi^{-1}(\mathscr{I}_1 \cdots \mathscr{I}_m) \cdot \mathscr{O}_X$
is locally free, and this defines a map from $X$ to $Y$.

For the reverse map, we note that the product $\rho^{-1}(\mathscr{I}_1) \cdots
\rho^{-1}(\mathscr{I}_m)\cdot \mathscr{O}_Y$ is locally free on $Y$, so each
factor $\rho^{-1}(\mathscr{I}_i) \cdot \mathscr O_Y$
must be invertible.  Thus, by the universal property
of blowing up, $Y \rightarrow X_{i-1}$ inductively lifts to $Y \rightarrow X_i$. This
defines a map $Y \rightarrow X$. Thus we have maps between $Y$ and $X$ which are
isomorphisms over the complement of the linear spaces $W_i$, so they
must be isomorphisms.
\end{proof}

\begin{thm}\label{thm:primary-components}
A projective variety $X$ arises as a primary component of the special fiber~$\M(\Gamma)_k$ for
some configuration $\Gamma$ of $n$ lattice points in
the Bruhat-Tits building $\mathfrak{B}_d$ if and only if $X$
is the blow-up of $\PP^{d-1}_k$ at a collection of $n-1$ linear subspaces.
\end{thm}

Before proving this theorem, we shed some light on the only-if direction by
describing the linear spaces in terms of the configuration $\Gamma$.
Fix an index $i$ and let  $C$ be the primary component of $\M(\Gamma)_k$ corresponding
to the lattice class $[L_i]$. 
 For any other point $[L_j]$ in $\Gamma$
we choose the unique representative~$L_j$ such that $L_j \supset \pi L_i$
but  $L_j \not\supset L_i$. Then the image of
$L_j \cap L_i$ in the quotient~$L_i/\pi L_i$ is a proper, non-trivial $k$-vector
subspace, and we denote by $W_j$ the corresponding linear subspace in $\PP(L_i)_k$.
The component~$C$ is the blow-up of $\PP(L_i)_k$ at the linear
subspaces $W_j$ for all $j \neq i$.

Since $[L_i]$ and~$[L_j]$ are in a common apartment, there is
a basis $e_1, \ldots, e_d$ of $V$ such that $L_i =  R\{ e_1,\ldots,e_d\}$
and  $L_j = R \{ \pi^{-s_1} e_1, \ldots, \pi^{-s_d} e_d \}$, where
$-1 = s_1 \leq \cdots \leq s_d$,
in order to satisfy the above condition on the representative.
Then $W_j$ is the linear space spanned by 
$\{e_i \,:\, s_i \geq 0\}$. In particular, if $\{W_i, W_j\}$ are in general position then
 $W_j$ is the hyperplane spanned by $e_2, \ldots, e_d$ and the blow-up of $W_j$ is trivial.

\begin{ex} \label{ex:delpezzo} \rm
Various classical varieties arise as primary components of some $\M(\Gamma)_k$. 
For instance,  any del Pezzo surface (other than
$\PP^1 \times \PP^1$) is the blow up of $\PP^2$ at  $\leq 8 $ general points, and thus
arises for an appropriate
configuration $\Gamma \in \mathfrak{B}_3$ with~$|\Gamma| \leq 9$. 
\qed
\end{ex}

\begin{proof}[Proof of Theorem~\ref{thm:primary-components}]
First we suppose that $\Gamma$ has only two elements, and we
choose  coordinates as in the discussion prior to Example \ref{ex:delpezzo}.
By Theorem \ref{thm:twistedVeronese}, we can compute the special
fiber from the arrangement of two tropical hyperplanes. We represent a point in $A = \R^d/\R(1,\ldots,1)$
by the last $d-1$ entries of a vector in $\mathbb R^d$, after rescaling
so that the first entry is~$0$. Thus, our tropical hyperplanes are centered
at~$(0, \ldots, 0)$
and $(s_2 + 1, \ldots, s_d + 1)$. The former point lies in the relative interior of the
cone of the latter tropical hyperplane generated by  $-e_t, \ldots, -e_d$,
where $t$ is the smallest index such that $s_t \geq 0$. This 
containment creates a ray at
$(0,\ldots, 0)$ generated by the vector $e_t + \cdots + e_d$, together with
adjacent cones. The resulting complete fan corresponds to the toric blow-up of $\PP^{d-1}_k$ at
the linear space spanned by $e_t, \ldots, e_d$.
This agrees with the description given 
after the statement of Theorem~\ref{thm:primary-components}.

Now suppose $\Gamma$ has $n > 2$ elements. We fix one element $[L_i]$.
Let $C$ denote the corresponding primary component of $\M(\Gamma)_k$. 
We claim that $C$ is
the blow-up of~$\PP(L_i)_k$ at the
linear subspaces~$W_j$ described after the theorem.
For each $[L_j] \in \Gamma \backslash \{ [L_i]\}$,
we have a projection $\M(\Gamma)\rightarrow \M(\{[L_i], [L_j]\})$ which sends $C$ to 
the $[L_i]$-primary component of $\M(\{[L_i], [L_j]\})_k$, and we denote
this component by $C_{j} \subset \PP(L_i)_k \times \PP(L_j)_k$.
We have shown $C_j$ to be the blow-up of~$\PP(L_i)_k$ at~$W_j$.
By taking the fiber product of these components with the
base $\PP(L_i)_k$ for all $j \neq i$, we get  the closed immersion
\begin{equation*}
C \,\rightarrow\, \prod_{ j \neq i}  C_{j} \, \rightarrow \, \prod_{ j =1}^n \PP(L_j)_k,
\end{equation*}
where the first product is the fiber product over $\PP(L_i)_k$ and the second is
over $\Spec k$.  
  Since each projection $C_{j} \rightarrow \PP(L_i)_k$ is a
birational morphism, the fiber product $\prod_{j \neq i}C_{j}$ contains 
an open subset which is mapped isomorphically to
 $\PP(L_i)_k$. Since $C$ is irreducible and 
birational with $\PP(L_i)_k$, we conclude that $C$ 
is isomorphic to the closure of this open set,
which is necessarily the desired primary component.

Let $B$ be the blow-up of $\PP(L_i)_k$ at the linear
subspaces $W_j$. We wish to show that $C$ is isomorphic to~$B$. Since $B$
maps compatibly to each $C_j$, we have a map to the fiber product $\prod_{j
\neq i} C_j$ with base $\PP(L_i)_k$. Since $B$ is irreducible and birational
with $\PP(L_i)_k$, this map factors through a map
 $B \rightarrow C$. On the other hand, the pullbacks of the ideal sheaves of each $W_j$ to~$C$ are all invertible, and this gives the inverse map from $C$ to $B$.
We conclude that $C$ and $B$ are isomorphic.

Finally, for any arrangement of $n-1$ linear subspaces in $\PP_k^{d-1}$
we can choose a configuration $\Gamma \in \mathfrak{B}_d$ with $n=|\Gamma|$
which  realizes the blow-up at these linear spaces
as a primary component. To do this, we represent each
linear space as a vector subspace $W_j$ of $k^d$, and 
we let $M_j$ denote the preimage in $R^d$ under the
residue map $R^d \rightarrow k^d$.
Then we take our configuration to be the standard lattice
 $R^d$ and the adjacent lattices $M_j + \pi R^d$. 
 The component of $\M(\Gamma)_k$ corresponding to $R^d$ is the desired blow-up.
\end{proof}

It follows from Theorem~\ref{thm:primary-components} that the
primary components are always smooth for $d \leq 3$ or $n=2$.
However, if $d \geq 4$ and $n \geq 3$, then we encounter 
 primary components that are not smooth.
These arise from the fact that the simultaneous
blow-up of projective space at several linear subspaces
may be singular. This was demonstrated in Figure \ref{fig:exploding}.

\begin{example} \label{ex:NonsimplePolytope}
Let $V = K^4$ with basis $\{e_1,e_2,e_3,e_4\}$ and
let $\Gamma = \{[L_1], [L_2], [L_3]\}$ be given by
$L_1 = R \{e_1,e_2,e_3,e_4\}$,
$L_2 = R \{e_1, \pi e_2,e_3, \pi e_4\}$ and
$L_3 = R \{e_1,e_2,\pi e_3,\pi e_4\}$.
The primary component
$C$ corresponding to $[L_1]$ is singular. It is obtained
from $\PP^3_k$ by blowing up the two lines spanned
 by~$\{e_1, e_2\}$ and
$\{e_1, e_3\}$ respectively.
This configuration was studied in
Example \ref{ex:explodingtetrahedron}. Its special fiber is drawn
in  Figure \ref{fig:exploding}, in which the polytope corresponding to~$C$ is
 the green cell. This polytope has a vertex that is adjacent to four edges.
\qed
\end{example}

Example \ref{ex:NonsimplePolytope}
underscores the fact that the convex configurations
considered by Mustafin \cite{mus} and Faltings \cite{fa} are
very special. In the convex case, all primary components are smooth.
This is a consequence of the following more general result.

\begin{prop}\label{thm:nice-blowups}
Fix the lattice $L_n$ and the linear spaces
$W_1$, \ldots, $W_{n-1}$ in $L_n/\pi L_n$ as after Theorem~\ref{thm:primary-components}.
Suppose
that for any pair of linear spaces $W_i$ and~$W_j$ either they intersect
transversely or their intersection $W_i \cap W_j$ equals some other $W_k$.
Moreover, we assume that they are ordered in a way refining inclusion,
so that $W_j \subset W_i$ implies $j < i$.
Then the primary component $C$ corresponding to $[L_n]$ is 
formed by successively blowing up the strict transforms of $W_1, \ldots,
W_{n-1}$ in that order. In particular, $C$ is smooth.
\end{prop}

\begin{proof}
We know that $C$ is formed by the successive blow-ups of the weak transforms of
the~$W_i$, so
we just need to show that these are equivalent to the blow-ups of the strict
transforms. Let $B_{i-1}$ be the blow-up of the strict transforms of
$W_1,\ldots, W_{i-1}$, and by induction, we assume this to be equal to the
blow-up of the weak transforms.
We will use $W_{i,j}$ and $\tilde W_{i,j}$ to denote the weak transform and
strict transform, respectively, of~$W_i$ in~$B_j$.

We claim that the weak transform $W_{i,i-1}$ is the union of
the strict transform $\tilde W_{i,i-1}$ with some exceptional divisors, which we
prove by tracing it through previous blow-ups.  For the $j$th step, in which we blow up  $\tilde
W_{j,j-1}$, we have three cases. First, if the original linear space $W_j$ is
contained in~$W_i$, then $W_{i,j}$ consists of the
strict transform of $W_{i,j-1}$ together with the exceptional divisor of the
blow-up. Second, if $W_j$ intersects $W_i$ transversely, then the weak transform of $W_{i,j-1}$
is equal to the strict transform. Third, if neither of the two previous cases hold,
then, by assumption, we must have already blown up $W_i \cap W_j$. In this case,
$\tilde W_{j,j-1}$ and $\tilde W_{i,j-1}$ are disjoint, so $\tilde W_{j,j-1}$
only intersects $W_{i,j-1}$ along the exceptional divisors of previous blow-ups, and these
intersections are transverse, so again, the strict transform of $W_{i,j-1}$ and
the weak transform coincide.

Therefore, $W_{i,i-1}$ consists of the strict transform~$\tilde W_{i,i-1}$
together with the exceptional divisors of the blow-ups of those $W_j$ which are
contained in~$W_i$.  Since these exceptional
divisors are defined by locally principal ideals, we can remove them without
changing the blow-up, so $B_i$ is isomorphic to the blow-up of~$\tilde
W_{i,i-1}$.  Since each strict transform $\tilde W_{i,i-1}$ is smooth, its
blow-up is also smooth, so $C$ is smooth.
\end{proof}

\begin{ex} \rm
The compactification $\overline{\mathcal{M}}_{0,m}$ of the moduli space of
$m$ points in~$\PP_k^1$
arises from  $\PP_k^{m-3}$ by blowing up $m-1$ general points followed by
blowing up the strict transforms of all linear spaces spanned by these points, in
order of increasing dimension~\cite[Theorem~4.3.3]{Ka}. Using
Proposition~\ref{thm:nice-blowups}, there exist configurations $\Gamma_m \in
\mathfrak{B}_d$ such that $\overline{\mathcal{M}}_{0,m}$ is a component of $\M(\Gamma_m)_k$. \qed
\end{ex}

The isomorphism types of the secondary components
of the special fibers are less restricted
than that of the primary components, but they
are still rational varieties.

\begin{lem}\label{lem:secondary-component}
Let $C$ be a secondary component in~$\M(\Gamma)_k$. There exists a vertex~$v$
in~$\mathfrak B_d^0$ such that if $\,\Gamma' = \Gamma \cup \{v\}$, then
$\M(\Gamma')\rightarrow \M(\Gamma)$ restricts to a birational morphism $\tilde C
\rightarrow C$, where $\tilde C$ is the primary component 
of $\M(\Gamma')_k$ corresponding to~$v$.
\end{lem}

\begin{proof}
Let $\overline\Gamma$ be the set of all
vertices in the convex closure of~$\Gamma$. By Lemma~\ref{lem:projection},
there is some component
of~$\M(\overline\Gamma)_k$ mapping birationally onto~$C$. By
Theorem~\ref{sec:must;thm:convex}, the component of~$\M(\overline\Gamma)_k$ must
be primary, and so corresponds to some vertex~$v$. Let $\Gamma' = \Gamma \cup
\{v\}$ and let $\tilde C$ be the
primary component corresponding to~$v$. Since $\M(\overline\Gamma) \rightarrow
\M(\Gamma)$ factors through $\M(\Gamma')$, $\tilde C$ must map
birationally onto~$C$.
\end{proof}

\begin{cor}\label{cor:secondary-rational}
Every secondary component is a rational variety.
\end{cor}

\begin{proof}
This follows from Lemma~\ref{lem:secondary-component} and
Theorem~\ref{thm:primary-components}.
\end{proof}

Now we wish to describe the geometry of the secondary components in more detail.
If $C$ is a secondary component of~$\M(\Gamma)_k$, we let
$C$, $\tilde C$ and~$\Gamma'$ be as in
Lemma~\ref{lem:secondary-component}, and we further let $\pi$ denote the projection
$\M(\Gamma') \rightarrow \M(\Gamma)$.
We identify $\tilde C$ with the blow-up of~$\PP_{k}^{d-1}$ at
the linear spaces~$W_i$ as in Theorem~\ref{thm:primary-components}.

\begin{lem}\label{lem:contraction}
If $\tilde L \subset \tilde C$ is the strict transform of a line $L$ in $\PP^{d-1}_k$,
then $\pi |_{\tilde L}$ is either constant or a closed immersion. Moreover, $\pi
|_{\tilde L}$ is constant if and only if $L$ intersects all  subspaces
$W_i$. 
The restriction of $\pi$ to an exceptional divisior is a closed immersion.
\end{lem}

\begin{proof}
Let $L$ be a line in $\PP^{d-1}_k$ and $\tilde L$ its strict transform.
Consider any vertex $w_i\in \Gamma$ with $W_i$ the corresponding linear space in
$\PP^{d-1}_k$. Then the projection of $\tilde L$ onto the $i$th factor 
$\PP^{d-1}_k$ is constant if $L$ intersects $W_i$ and is a closed immersion if not.
Since the projection of $\tilde L$ to $\M(\Gamma)_k$ consists of the projection to
the fiber product of these factors, we have the desired result.
For any exceptional divisor, the projection to the factor of~$v$ is constant, so
the projection to $\M(\Gamma)_k$ must be a closed immersion.
\end{proof}

\begin{prop}\label{prop:exceptional}
With the set-up as in Lemma~\ref{lem:contraction}, the exceptional locus of
$\tilde C \rightarrow C$ is the union of the strict transforms of all lines
which intersect all of the subspaces~$W_i$.
\end{prop}

\begin{proof}
If a line $L \subset \PP^{d-1}_k$ passes through all of the $W_i$, then its strict transform in $\tilde
C$ is contracted in $C$ by Lemma~\ref{lem:contraction}, so any point on the line
is in the exceptional locus.
Conversely, suppose that $x$ is in the exceptional locus, so there exists a
point $y$ in~$C$ such that $\pi(x) = \pi(y)$. By Lemma~\ref{lem:contraction},
$x$ and~$y$ cannot be in an exceptional divisor. Thus, we take the projections
of $x$ and~$y$ to $\PP^{d-1}_k$ and let $L$ be the line through them.
Lemma~\ref{lem:contraction} implies that $\tilde L$ must be contracted to a
point by $\pi$, and therefore $L$  intersects all the linear spaces $W_i$.
\end{proof}

\begin{example} \label{ex:hexagon}
Fix a basis $e_1, e_2, e_3$ of~$V$. Let $L_i = R\{\pi e_1, \pi
e_2, \pi e_3, e_i\}$ and consider the secondary component
associated to $v = [R\{e_1, e_2, e_3\}]$. Thus, $W_i = ke_i$.
If $\Gamma = \{[L_1], [L_2]\}$, then $C$ is the
blow-up of $\PP_k^2$ at the points $W_1$ and $W_2$, followed by the blow-down of
the line between them, yielding $\PP_k^1 \times \PP_k^1$. If $\Gamma =
\{[L_1], [L_2], [L_3]\}$, then $\tilde C$ and $C$ are both isomorphic to the
blow-up of $\PP^2_k$ at the three points $W_1$, $W_2$ and~$W_3$, since there
are no lines passing through all three points. This is the same configuration as
in Example~\ref{ex:HexagonWithThreeTriangles}. \qed
\end{example}

\begin{example} \label{ex:sailboat}
Suppose that $L_1 = R\{e_1, \pi e_2, \pi e_3\}$, $L_2 = R\{\pi e_1, e_2,
\pi e_3\}$ and $L_3 = R\{e_1 + e_2, \pi e_2, \pi e_3\}$, and let $\Gamma =
\{[L_1], [L_2], [L_3]\}$. Then $\M(\Gamma)_k$ has
a secondary component that is singular, indexed by
 $L_4 = R\{e_1, e_2, e_3\}$. The
corresponding primary component of $\M(\Gamma \cup \{[L_4]\})_k$ is the 
blow-up of~$\PP_k^2$ at three collinear points
$W_1 = k e_1$, $W_2 = ke_2$, $W_3 = k (e_1+e_2)$.
The secondary component in~$\M(\Gamma)_k$ arises by blowing down the
strict transform of the line through these three points.
Algebraically, the ideal
\begin{equation}
\label{eq:sailboat}
\langle x_1, z_1, y_2, z_2 \rangle \cap
\langle x_1, z_1, x_3, z_3 \rangle \cap
\langle y_2, z_2, x_3, z_3 \rangle \cap
\langle x_1, y_2, x_3, z_1z_2 y_3 + z_1x_2z_3 - y_1 z_2 z_3 \rangle
\end{equation}
represents $\M(\Gamma)_k$, where
$(x_i : y_i : z_i)$ are the coordinates on the 
$i$th factor of $\PP^2 {\times} \PP^2 {\times} \PP^2$. 
The last prime ideal in (\ref{eq:sailboat}) is the secondary component.
It has a quadratic cone singularity at the point $\bigl((0\!:\!1\!:\!0), (1\!:\!0\!:\!0), (0\!:\!1\!:\!0)\bigr)$.
This special fiber looks like a {\em sailboat}:
the secondary component is the boat; its sails are  three
projective planes attached at three of its lines.
In the census of Theorem \ref{thm:TriangleCensus},
this is the unique type which is not a union of toric surfaces, so it cannot be drawn as a 
$2$-dimensional polyhedral complex.
\qed \end{example}

\section{Triangles}

In this section we classify Mustafin varieties for $d=n=3$. 
We refer to a triple $\Gamma$ in $\mathfrak{B}^0_3$ as a \emph{Mustafin triangle}.
Two such triangles are said to have the same {\em combinatorial type} if the special fibers of the associated Mustafin varieties are isomorphic as $k$-schemes.
Note that we introduced combinatorial types already after Proposition~\ref{prop:classify2}. Since this notion only involves the special fibers, it is different from the notion of isomorphisms of Mustafin varieties investigated in Section~2.

\begin{thm}\label{thm:TriangleCensus}
There are precisely $38$ combinatorial types of Mustafin triangles.
  In addition to the $18$ planar types 
(in Figure \ref{fig:18types}) there are $20$ non-planar types
(in Table~\ref{tbl:isomorphism-counts}).
\end{thm}

\begin{table}\begin{centering}\begin{tabular}{|c|c|c|c|c|}
\hline
& \multicolumn{4}{c|}{Number of bent lines} \\
Number of components & 0 & 1 & 2 & 3 \\
\hline
3  & {\it 2}\,+\,{\bf 0} & {\it 1}\,+\,{\bf 0} & &  \\
4  & &  {\it 3}\,+\,{\bf 3} & {\it 1}\,+\,{\bf 0} & {\it 1}\,+\,{\bf 1} \\
5  & &  & {\it 5}\,+\,{\bf 6} & {\it 0}\,+\,{\bf 2} \\
6  & &  &   & {\it 5}\,+\,{\bf 8} \\
\hline
\end{tabular}\par\end{centering}
\caption{Classification of  the
{\it 18 planar} and {\bf 20 non-planar} types of Mustafin triangles.}
\label{tbl:isomorphism-counts}
\end{table}

The term {\em non-planar} is used as in \cite[\S 5]{CS}. It refers to combinatorial types
consisting entirely of configurations $\Gamma$ that do not lie in a single apartment.
The rest of this section is devoted to proving Theorem~\ref{thm:TriangleCensus}.
Since the planar configurations were enumerated in Example~\ref{ex:planar-3-3},
our task is to classify all non-planar Mustafin triangles, and to show that all
types are realizable over any valuation ring $(R,K,k)$.

The convex hull of any two points $v$ and~$w$ in $\mathfrak{B}_3$ is a tropical line segment
which is contained in a single apartment. If this tropical line consists of a single
Euclidean line segment, then we say that the line is \emph{unbent}. Otherwise,
the tropical line consists of two Euclidean line segments and we call the line
\emph{bent} and the junction of the two lines the \emph{bend point}. 
Note that in this section the term \emph{line} always means tropical line. 
By the bend points of a configuration $\Gamma$
we mean the bend points of all pairs of points in~$\Gamma$.

\begin{prop}\label{prop:existence-secondarycomp}
For a finite subset $\Gamma $ of $\mathfrak{B}_3^0$,
the secondary components of $\M(\Gamma)_k$ are in bijection with the bend points
of $\Gamma$ which are not themselves elements of the set $\Gamma$.
\end{prop}

\begin{proof}
Let $\Gamma = \{v_1,\ldots, v_n\}$ and let
$\Gamma'$ be the union of $\Gamma$ and the set of bend points of~$\Gamma$.
By Lemma~\ref{lem:projection}, each component $C$ of $\M(\Gamma)_k$ is the image of a 
{\em unique} component $C'$ of~$\M(\Gamma')_k$ under the natural projection. 
Suppose that $C$ is secondary in  $\M(\Gamma)_k$. To establish the bijection,
we must prove that $C'$ is primary in $\M(\Gamma')_k$.

Recall from the proof of
Lemma~\ref{lem:projection} that the cycle class of any component of
$\M(\Gamma)_k$ is a sum of distinct monomials of $\Z[H_1, \ldots, H_n]/
\langle H_1^3, \ldots, H_n^3 \rangle$ having degree $2n-2$. Since $C$ is not a
primary component, every term in this cycle class involves every variable.
Therefore, after permuting the factors, we can assume that the cycle class
of~$C$ contains the term $H_1 H_2 H_3^2 \cdots H_n^2$. The image of~$C$
under the projection $\M(\Gamma) \rightarrow \M(\{v_1, v_2\})$ must be a
component, and since $C$ is a secondary component, so is its image. The vertices
$v_1$ and~$v_2$ lie in a common apartment, and Theorem~\ref{thm:twistedVeronese}
implies that the tropical convex hull of $v_1$ and~$v_2$ must have a bend point,
which we denote $w$,
corresponding to the secondary component. By
Lemma~\ref{lem:projection}, the secondary component is the image of some
component of $\M(\{v_1, v_2, w\})_k$. Using
Theorem~\ref{thm:twistedVeronese} we see that the only candidate is the
$w$-primary component.  The $w$-primary component
of~$\M(\Gamma')_k$ maps onto the $w$-primary component of $\M(\{v_1, v_2,
w\})_k$ and then onto the secondary component of $\M(\{v_1, v_2\})_k$, so it 
equals the unique component $C'$ that maps surjectively onto~$C$.
\end{proof}

Our classification of non-planar Mustafin triangles will proceed in two phases.
For special fibers with few components, the key result is
Lemma~\ref{lem:two-unbent} below. On the
other hand, for special fibers with five or six components,
Lemma~\ref{lem:good-coords} will imply that that their ideals 
are monomial or ``almost monomial''. Before getting to these technical
phases, however, we first discuss all bold face entries in Table
\ref{tbl:isomorphism-counts}, starting with the last column.

\begin{example}[Mustafin triangles with three bent lines]
\label{ex:three-bent}
Let $\Gamma $ be a triple in $\mathfrak{B}_3^0$ with three bent lines.
The first row of Table \ref{tbl:isomorphism-counts} concerns types
without any secondary component, which is not possible if there are three bent
lines.
In the second row we find types
with one secondary component. There is one planar possibility, namely,
the type  \includegraphics[scale=.2]{typ-c1-N.eps},
and one non-planar possibility, namely the ``sailboat'' in
Example \ref{ex:sailboat}.
Corollary \ref{cor:six-comps} will take care of the last row in
Table \ref{tbl:isomorphism-counts}: these are the $13$ monomial ideals
in \cite[Table 1]{CS} that lie on the main component of the
Hilbert scheme $H_{3,3}$. In the pictures offered in \cite[Figure 2]{CS}
we recognize the {\it 5} planar monomial types
      \includegraphics[scale=.19]{typ-a1-N.eps} 
      \includegraphics[scale=.19]{typ-a2-N.eps} 
      \includegraphics[scale=.19]{typ-a3-N.eps} 
      \includegraphics[scale=.19]{typ-a4-N.eps} 
      \includegraphics[scale=.19]{typ-a5-N.eps},
      and the {\bf 8} non-planar types are obtained from these
   by {\em regrafting triangles}.

 An especially interesting entry in Table \ref{tbl:isomorphism-counts}  
 is the rightmost entry ${\it 0} + {\bf 2}$ in the third row.
    No planar types have two secondary components and three bent lines,
 but there are two non-planar types. Their pictures are
 shown in Figure~\ref{fig:airplanes}. The three lattices
\begin{align}\label{eqn:airplane-config}
L_1 &= R \{\pi e_1, \pi e_2, e_3\},  &
L_2 &= R \{e_1, \pi^2 e_2, \pi^2 e_3\},  &
L_3 &= R \{e_1 + \pi e_2, \pi^2 e_2, \pi^2 e_3\}
\end{align}
give a Mustafin variety whose special fiber is defined by the ideal
\begin{equation}\label{eqn:airplane}
\langle y_1,\! z_1,\! x_2,\! z_2 \rangle \cap
\langle y_1,\! z_1,\! x_3,\! z_3 \rangle \cap
\langle x_2,\! \underline{y_2},\! x_3,\! \underline{y_3} \rangle \cap
\langle y_1,\! z_1,\! x_2,\! x_3 \rangle \cap
\langle \underline{z_1},\! x_2,\! x_3,\! z_2 y_3 - y_2 z_3 \rangle.
\end{equation}
The primary components of this special fiber are all isomorphic to $\PP^2_k$ and
embedded in~$\PP^2 {\times} \PP^2 {\times} \PP^2$ as coordinate linear spaces. The two secondary components are
isomorphic to $\PP^1_k {\times} \PP^1_k$. One of these copies of $\PP^1_k {\times}
\PP^1_k$ is embedded as a coordinate linear space. The other is embedded as
a coordinate linear space times the diagonal of the previeous secondary component.
Thus, the intersection of the two secondary components is the diagonal in the
first and a line of one of the rulings in the second. Two of the primary
components are attached along coordinate lines of the former secondary
component.  The final primary component is glued along the unique coordinate
line of the diagonal secondary component which does not intersect the other two
primary components.

Both initial ideals of~(\ref{eqn:airplane}) belong to isomorphism class~11 from~\cite[\S 5]{CS}
which is the fifth picture in the second row of \cite[Figure 2]{CS}.
The special fiber~(\ref{eqn:airplane}) is obtained from that  picture 
 by removing the two uppermost parallelograms and
replacing them with a long rectangle which is attached to the diagonal of the lower
parallelogram.

\begin{figure}
\centering
\includegraphics[width=7.0cm]{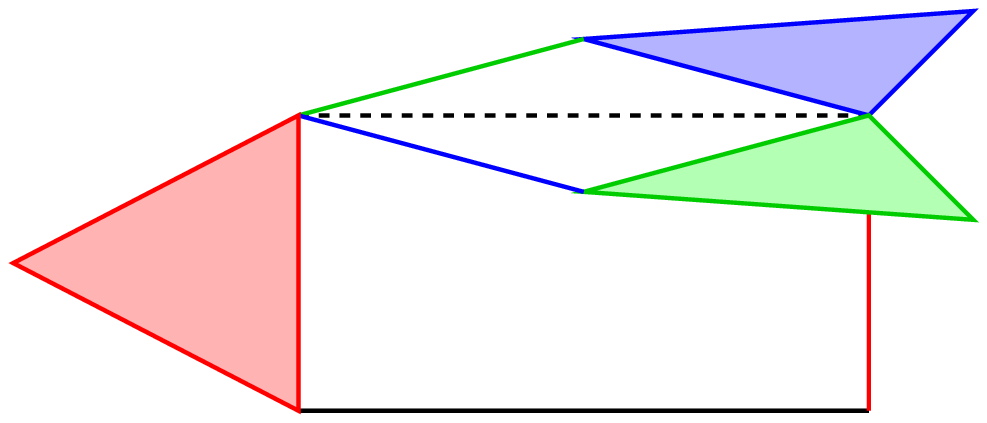} \qquad 
\includegraphics[width=6.6cm]{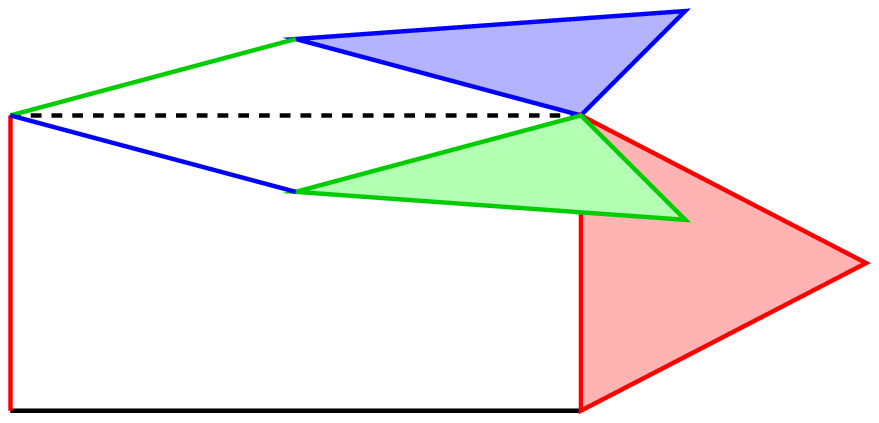} 
\caption{
The two combinatorial types of Mustafin triangles with three bent lines and two secondary components. The ideal (\ref{eqn:airplane}) is on the left while (\ref{eqn:mutant-airplane})
is on the right.}
\label{fig:airplanes}
\end{figure}

The other non-planar special fiber for the entry {\it 0} + {\bf 2} of
Table \ref{tbl:isomorphism-counts} is the variant
of the above configuration by taking 
$L_1' = \{\pi e_1, \pi e_3, e_2\}$
instead of $L_1$ in~(\ref{eqn:airplane-config}). Its ideal is
\begin{equation}\label{eqn:mutant-airplane}
\langle y_1,\! z_1,\! x_2,\! z_2 \rangle \cap
\langle y_1,\! z_1,\! x_3,\! z_3 \rangle \cap
\langle x_2,\! \underline{z_2},\! x_3,\! \underline{z_3} \rangle \cap
\langle y_1,\! z_1,\! x_2,\! x_3 \rangle \cap
\langle \underline{y_1},\! x_2,\! x_3,\! z_2 y_3 - y_2 z_3 \rangle.
\end{equation}
The components of~(\ref{eqn:mutant-airplane}) are identical to those
of~(\ref{eqn:airplane}) except that the final primary component is glued along
the other coordinate line of the diagonal secondary component, so that all three
primary components intersect.  Its initial ideals are of isomorphism class~13
from \cite[\S 5]{CS}, which is the left picture in the third row of \cite[Figure 2]{CS}.
The special fiber~(\ref{eqn:mutant-airplane}) is again obtained 
 by replacing the two uppermost parallelograms
with a long rectangle which is attached to the diagonal of the lower
parallelogram.
 \qed
\end{example}

We now come to the first technical phase in our classification of Mustafin triangles.

\begin{lem}\label{lem:two-unbent}
Let $[L_1]$, $[L_2]$ and~$[L_3]$ be distinct points in $\mathfrak B_3^0$ such
that the line between $[L_1]$ and~$[L_i]$ is unbent for $i=2,3$. Either the three points lie in a single apartment, or there exists a
basis $e_1, e_2, e_3$ of $L_1$ and integers $0 < t < s, u$
such that the other lattices can be written in one of the following forms (possibly after exchanging $[L_2]$ and $[L_3]$):
\begin{enumerate}
\item $L_2 = R\{e_1, \pi^s e_2, \pi^s e_3\}$ and
$L_3 = R\{e_1 + \pi^t e_2, \pi^u e_2, \pi^u e_3\}$,
\item $L_2 = R\{e_1, \pi^s e_2, \pi^s e_3\}$ and
$L_3 = R\{e_1 + \pi^t e_3, e_2, \pi^u e_3\}$,
\item $L_2 = R\{e_1, e_2, \pi^s e_3\}$ and
$L_3 = R\{e_1 + \pi^t e_3, e_2, \pi^u e_3\}$.
\end{enumerate}
\end{lem}

\begin{proof}
We can choose representatives $L_2 = M_2 + \pi^{s} L_1$ and $L_3 =
M_3 + \pi^{u} L_1$ such that $s, u > 0$ and the $R$-modules $M_i$
are direct summands of $L_1$.  The
$R$-module $M_2 \cap M_3$ is also a direct summand of $L_1$, and its rank is
smaller or equal to the $k$-dimension of $\overline M_2 \cap \overline M_3$.
Here $\overline{M_2}$ and
$\overline{M_3}$ are the subspaces of
$L_1 / \pi L_1 \simeq k^3$ induced by $M_2$ and $M_3$.  

Suppose the rank of $M_2 \cap M_3$ equals the
$k$-dimension of $\overline M_2 \cap \overline M_3$. We claim that 
the lattices are in a single apartment.  We pick a $k$-basis
$\overline e_1, \overline e_2, \overline e_3$ for $L_1/\pi L_1$ such that
 $\overline M_2$, $\overline M_3$ and $\overline M_2 \cap \overline M_3$ are all
spanned by subsets of this basis. By our assumption on the dimension, we have
$\overline{M_2 \cap M_3} = \overline{M_2} \cap \overline{M_3}$. Hence we can
lift the basis elements in $\overline M_2 \cap \overline M_3$ to $M_2 \cap M_3$.
We also lift the other basis elements, choosing representatives in $M_2$
and~$M_3$ for the elements in $\overline{M_2}$ and $\overline{M_3}$,
respectively. By Nakayama's Lemma, $L_1$ is generated by the lifts of all
three basis elements, and
$M_2$ and $M_3$ are generated by subsets thereof. Thus, 
$\{[L_1],[L_2],[L_3]\}$ lies
in the apartment defined by this basis.

We now assume that the rank of $M_2 \cap M_3$ is strictly smaller
than the $k$-dimension of $\overline M_2 \cap \overline M_3$. Note that the
ranks of $M_2$ and $M_3$ are either one or two.  If $M_2$ and $M_3$ both have
rank one, our assumption implies  that  $M_2 \cap M_3 = 0$ and 
${\rm dim}_k(\overline{M_2} \cap \overline{M_3}) = 1$, so that $\overline{M_2}
= \overline{M_3}$.  Let $e_1$ be a generator of $M_2$. We lift  $\overline{e_1}$
to a generator   $e_1 + \pi^t e_2$  of $M_3$, where $t
\geq 1$ and $e_2 \in L_1 \backslash \pi L_1$. Since $M_2 \cap M_3 = 0$, we can assume
that $\overline{e_1}, \overline{e_2}$ are linearly independent. Hence they can be 
 completed to a basis of $L_1 / \pi L_1$, which lifts  to
a basis  $e_1, e_2, e_3$ of $L_1$.  If $t \geq u$, then $L_3 = M_3 + \pi^u L_1 =
R\{e_1, \pi^u e_2, \pi^u e_3\}$, and all three lattice classes lie in the
apartment given by $e_1,e_2,e_3$.  If $t \geq s$, then $L_2  = M_2 + \pi^s L_1 =
R\{e_1 + \pi^t e_3, \pi^s e_2, \pi^s e_3\}$ and all three lattice classes lie in the apartment 
given by $e_1 + \pi^t e_3, e_2, e_3$.  If $t < s$ and $t<u$, then we are in case (i).

If $M_2$ has rank one and $M_3$ has rank two, the dimension of $\overline{M_2}
\cap \overline{M_3}$ is at most one.  Our assumption implies that it is
one and that $M_2 \cap M_3 = 0$. Let $e_1$ be a generator for $M_2$. We
fix $e_2 \in M_3$ such that $\overline{e_1},
\overline{e_2}$ is a basis of $\overline{M_3}$.
We choose  $e_3'$ to complete $e_1$ and $e_2$ to a basis for $L_1$. Then
$M_3$ is generated by $e_2$ and an element of the form $e_1 + \pi^t u
e_3'$, where
$u$ is a unit in~$R$.
By replacing $e_3'$ with $e_3 = u e_3'$ in our basis, $M_3$ is generated
by $e_2$ and $e_1 + \pi^t e_3$.
If $t \geq s$ or $t \geq u$, the three lattice classes lie in one apartment by the same argument as in the previous case, and 
if $t<s$ and $t<u$, we are in case (ii).

If $M_2$ and $M_3$ both have rank two, then $M_2 \cap M_3$ has rank one, since 
$M_2 \not= M_3$.  Our assumption implies that  $\overline{M_2}
\cap \overline{M_3}$ is two-dimensional, so that $\overline{M_2} =
\overline{M_3}$. Choose a generator $e_2$ of $M_2 \cap M_3$. Since $M_2 \cap
M_3$ is a split submodule of $M_2$, we can complete it to a basis $e_1,e_2$ of
$M_2$. As in the previous case,
we choose some $e_3'$ such that $e_1, e_2, e_3'$ is a basis for $L_1$. Then
$M_3$ can be generated by $e_2$ and an element of the form $e_1 + \pi^t u
e_3'$, where
$u$ is a unit in~$R$.
By replacing $e_3'$ with $e_3 = u e_3'$ in our basis, $M_3$ is generated
by $e_2$ and $e_1 + \pi^t e_3$.
Hence we are in case
(iii) if $t < u$ and $t <s$. Otherwise, the same argument as above shows that
the three lattice classes  lie in one apartment.  \end{proof}

\begin{cor}\label{cor:three-unbent}
If $\Gamma$ has no bent lines, then $\Gamma$ lies
in a single apartment.
\end{cor}

\begin{proof}
The exceptional cases in Lemma~\ref{lem:two-unbent} each have a bent line
between $L_2$ and~$L_3$.
\end{proof}

The Chow ring of the
product $\PP^2
\times \PP^2 \times \PP^2$ is $\Z[H_1, H_2, H_3]/\langle H_1^3, H_2^3, H_3^3\rangle$.
As seen in the proof of Lemma~\ref{lem:projection}, 
 the special fiber $\M(\Gamma)_k$ 
of a Mustafin triple $\Gamma$ has the class
\begin{equation}\label{eqn:class-triangle} H_2^2H_3^2 + H_1^2H_3^2 + H_1^2H_2^2
+ H_1H_2H_3^2 + H_1H_2^2H_3 + H_1^2 H_2H_3.  \end{equation} The cycle class of
each component of $\M(\Gamma)_k$ is a sum of a subset of these monomials.

\begin{lem}\label{lem:good-coords}
There exist coordinates on $\PP^2 \times \PP^2 \times \PP^2$ such that the
projection of each component of $\M(\Gamma)_k$ is a coordinate linear space.
Each component whose cycle class is one of the monomials
in~(\ref{eqn:class-triangle}) is a product of coordinate linear spaces in this
basis. \end{lem}

\begin{proof}
Let $\Gamma = \{v_1, v_2, v_3\}$ and $C_1,C_2,C_3$ 
the corresponding primary components of~$\M(\Gamma)_k$.
Let $C_{12}$ denote the unique 
component whose cycle class includes the monomial $H_1H_2H_3^2$ and similarly
for the other pairs of indices. In this manner,
every component of $\M(\Gamma)_k$ has
a label, which may not be unique, e.g.~$C_1 = C_{12}$ or $C_{12} = C_{13}$ are
allowed.

For each component~$C$ other than $C_1$, we project to
$\M(\{v_1\})_k \simeq \PP^2_k$. If the image of $C$ is positive-dimensional then it
meets every line  in $\PP^2_k$, so $H_1 \cdot [C]$ is a non-zero cycle.
In this case, $C$ must be $C_{12}$ or $C_{13}$, because the other monomials
in~(\ref{eqn:class-triangle}) are annihilated by $H_1$. Thus, each component
other than $C_1$, $C_{12}$ and~$C_{13}$ maps to a point in~$\M(\{v_1\})_k
\simeq \PP^2_k$. We shall describe the images of these components more explicitly.

Since $v_1$ and~$v_2$ lie in a single apartment,
Theorem~\ref{thm:twistedVeronese} implies that $\M(\{v_1, v_2\})_k$ is either
the union of two copies of $\PP_k^2$ and $\PP_k^1{\times} \PP_k^1$, or the union of
$\PP_k^2$ and the blow-up of $\PP_k^2$ at a point. Looking at the cycle classes,
we see that $C_{12}$ maps onto $\PP_k^1 {\times} \PP_k^1$ in the
first case and onto the blow-up in the second case.  If $C_{12}$ is
distinct from~$C_1$, then a copy of~$\PP^2_k$ in~$\M(\{v_1, v_2\})_k$
maps isomorpically onto $\M(\{v_1\})_k$. Thus, the image of~$C_{12}$ in
$\M(\{v_1\})_k$ must be their line of intersection in $\M(\{v_1, v_2\})_k$.
Also by Lemma~\ref{lem:projection}, the image of~$C_2$ must be one of the
components in $\M(\{v_1, v_2\})_k$, so it meets the image
of~$C_{12}$. If $C_2 $ is different from $C_{12}$, then we already saw that its
image in
$\M(\{v_1\})_k$ is a point, and since its image intersects $C_{12}$, it must be
a point in the image of~$C_{12}$.
Suppose that $C_{23}$ has cycle class $H_1^2H_2H_3$, and so has no other label.
By Theorem~\ref{sec:must;thm:basicfacts}, $\M(\Gamma)_k$ is Cohen-Macaulay and
hence connected in codimension~$1$. The curves on
$C_{23}$ all have cycle classes which are linear combinations of $H_1^2H_2^2
H_3$ and $H_1^2 H_2 H_3^2$. Thus, $C_{23}$ intersects either $C_2$ or~$C_3$
in codimension~$1$. If the image of either $C_2$ or $C_3$ is a point, then the
image of $C_{23}$ must be the same point.

In conclusion, the set of images  in $\M(\{v_1\})_k \simeq \PP^2_k$
 of the components other than $C_1$ consists of at
most two lines and at most one point in each of the lines. We can always choose
coordinates such that each of these is a coordinate linear space. Repeating this
for each of the projections gives the desired system of coordinates  
on $\PP^2 {\times} \PP^2 {\times} \PP^2$.

For the second assertion of Lemma \ref{lem:good-coords},
we choose coordinates as above and let $C$ be a component
whose cycle class is a monomial in the Chow ring. Then $C$ must be the product
of linear subspaces, so $C$ is equal to the product of its images on the
projections, which must be coordinate linear spaces by the above argument.
\end{proof}

\begin{cor}\label{cor:six-comps}
A Mustafin triangle $\Gamma$ is of monomial type if and only if
the special fiber $\M(\Gamma)_k$ has six
irreducible components.
\end{cor}

\begin{proof}
Monomial type implies six irreducible components
by Remark~\ref{rmk:monomial-max-comps}. Conversely, if $\M(\Gamma)_k$ has six
irreducible components, then each component has a unique monomial
from~(\ref{eqn:class-triangle}) as its cycle class, so
Lemma~\ref{lem:good-coords} implies that $\M(\Gamma)_k$ is of monomial type.
\end{proof}

At this point in our derivation, the following facts about
Table~\ref{tbl:isomorphism-counts} have been proved. All entries below the main diagonal are
zero: by Proposition \ref{prop:existence-secondarycomp}, the number
of secondary components cannot exceed the number of bent lines.
Corollary \ref{cor:three-unbent} confirms the first column.
Lemma \ref{lem:two-unbent} confirms the second column.
Corollary \ref{cor:six-comps} shows that ${\it 5} + {\bf 8}$ is an upper bound
for the lower right entry. We now prove the matching lower bound.

\begin{prop}\label{prop:monomial-realizable}
Each of the $13$ isomorphism classes of
monomial ideals on the main component of the Hilbert scheme $H_{3,3}$  arises as
the special fiber of a Mustafin variety.
\end{prop}

\begin{proof}
For any integer vector $(a,b,c,d,e,f,g,h)$ consider the three matrices
\[\begin{pmatrix}  1 & 0& 0 \\ 0 & 1 & 0 \\ 0 & 0&1 \end{pmatrix}, \quad    G =
\begin{pmatrix}   1&  0 & 0  \\ 0 & \pi^a & 0 \\ 0 & 0 & \pi^b   \end{pmatrix}
\quad \text{ and } \quad  H=  \begin{pmatrix}\pi^c& \pi^d & (1+\pi) \pi^e  \\   0   &
\pi^f &  \pi^g  \\0 & 0  & \pi^h \end{pmatrix}\]
in $GL_3(K)$. Consider the configuration
$\Gamma= \{ [L_1], [GL_1], [HL_1]\}$ in the building~$\mathfrak B_3$
where $L_1 \simeq R^3$ is a reference lattice. The
generic fiber of the Mustafin variety $\M(\Gamma)$ is defined by the $2 \times
2$-minors of the matrix
\begin{equation}\label{eqn:realizability-matrix}
\begin{pmatrix}
         x_1 &    x_2  &  x_3 \pi^c+y_3 \pi^d+ z_3 (1 + \pi) \pi^e \\
         y_1 &  y_2 \pi^a &       y_3 \pi^f+z_3 \pi^g    \\
          z_1 &  z_2 \pi^b &            z_3 \pi^h      
           \end{pmatrix}
\end{equation}
The following list proves that each of the $13$ isomorphism classes
of monomial ideals in $H_{3,3}$ can be realized as a special fiber for some $\Gamma$.
We use the labeling in  \cite[Table~1]{CS}:

\begin{equation*} \begin{array}{l|l|l} 
\!\!\!\! \text{type}\! \! & (a,b,c,d,e,f,g,h) &  \text{monomial ideal defining the special fiber} \\
\hline
1  &  (-1, 1, 0, 1, 0, 1, 0, -1) &  \langle y_2 z_3, x_2 z_3, y_1 z_3, x_1 z_3, x_3 y_2, x_3 y_1,
y_2z_1, x_1y_2, x_2z_1\rangle \\
2 & (-1, 3, -1, 0, 1, 0, 1, 1) & \langle y_2z_3, x_2z_3, y_3z_1, x_3y_2, x_3z_1, x_3y_1,
y_2z_1, x_1y_2, x_2z_1\rangle \\
3 &  (-1, 2, -1, 0, 0, 1, 1, 0) & \langle y_2z_3, x_2z_3, y_1z_3, x_3y_2, x_3z_1, x_3y_1,
y_2z_1, x_1y_2, x_2z_1\rangle \\
4 & (-1, 1, 1, 1, 2, -1, 0, 0) & \langle y_2z_3, x_2z_3, x_1z_3, x_2y_3, y_3z_1, x_1y_3,
y_2z_1, x_1y_2, x_2z_1\rangle \\
5 & (1, -2, 1, 0, 2, 2, 4, 1) & \langle y_1z_3, y_3z_2, y_2y_3, y_3z_1, y_1y_3, x_3z_2,
z_2y_1, z_2x_1, x_2y_1, x_1y_2z_3\rangle \\
8 & (1, -2, 2, 0, 2, -1, 0, 0) & \langle z_2z_3, x_1z_3, y_3z_2, x_2y_3, y_3z_1, x_1y_3,
z_2y_1, z_2x_1, x_2y_1\rangle \\
9 & (1, 3, 1, 0, 0, 2, -1, 0) & \langle y_2z_3, x_2z_3, z_3z_1, x_1z_3, y_2y_3, y_3z_1,
y_2z_1, x_2z_1, x_2y_1\rangle \\
10 & (2, 1, 0, -1, -2, 1, -1, 0) & \langle z_2z_3, x_2z_3, z_3z_1, y_1z_3, y_3z_2, y_3z_1,
z_2y_1, x_2z_1, x_2y_1\rangle \\
11 & (-4, 1, 3, 4, 1, 1, 0, 3) & \langle z_2z_3, y_2z_3, z_3z_1, x_1z_3, y_3z_1, x_3y_2,
y_2z_1, x_1y_2, x_2z_1, y_3z_2x_1\rangle \\
13 & (-3, -6, 6, 7, 3, 4, 3, 0) & \langle z_2z_3, y_2z_3, y_1z_3, x_1z_3, y_3z_2, x_1
y_3, z_2y_1, z_2x_1, x_1y_2, x_3y_2y_1\rangle \\
14 & (3, 1, 0, 1, -1, 3, 1, 0) & \langle z_2z_3, x_2z_3, z_3z_1, y_1z_3, x_3z_2, x_3y_1,
z_2y_1, x_2z_1, x_2y_1\rangle \\
15 & \! (-1, 4, 1, -1, -2, 1, -1, 1) \! & \langle y_2z_3, x_2z_3, z_3z_1, y_1z_3, y_2y_3, y_3z_1,
y_2z_1, x_1y_2, x_2z_1, y_1x_2y_3\rangle \\
16 & (2, -2, 1, 0, 0, 1, 0, -2) &\langle z_2z_3, x_2z_3, y_1z_3, x_1z_3, y_3z_2, y_1
y_3, z_2y_1, z_2x_1, x_2y_1, y_3x_2x_1\rangle 
\end{array} \end{equation*}
We argue below that these realizations are independent of the choice of  $(R,K,k)$.
\end{proof}

\begin{remark} \label{rmk:compute} Our proof of
Proposition~\ref{prop:monomial-realizable} relies on the {\em computation} of the special
fiber of a Mustafin variety {\em over an arbitrary discrete valuation ring}. For this, we work over the
$2$-dimensional base ring $T = S^{-1} \Z[t]$, where $S$ is the multiplicative
set $\{1 + r \mid r \in t \cdot \Z[t]\}$.  We take the ideal of $2\times
2$-minors of~(\ref{eqn:realizability-matrix}), with $t = \pi$ and saturate with
respect to~$t$ to obtain an ideal $I$ in $T' = T[x_i,y_i,z_i : 1 \leq i \leq
3]$.  The essential check is that the resulting quotient ring $T'/I$ is flat
over~$S$. For any discrete valuation ring~$R$ with uniformizer~$\pi$, there is a unique homomorphism $f\colon S \rightarrow
R$ which sends $t$ to~$\pi$. Since the subscheme of~$(\PP_R^{2})^3$ defined by
$f(I)$ is flat over~$R$, it coincides with the desired Mustafin variety $\M(\Gamma)$.  The
special fiber is defined by the image of $I$ in $k \otimes_\Z (T'/\langle
t\rangle) = k[x_i, y_i, z_i : 1 \leq i \leq 3]$.

There are no well-developed tools for computing in~$T$ directly, so we use
\texttt{Singular} with a local term order in the ring $\Q \otimes T = \Q[t]_{(t)}$.
 We compute $I_\Q$ as the saturation
of the ideal of the $2 \times 2$ minors of~(\ref{eqn:realizability-matrix}) with
respect to~$t$. As long as our standard basis for~$I_\Q$ has integer
coefficients and each leading term has coefficient~$1$, we can
define $I$ in $T'$ to be the ideal generated by the same polynomials. Each
reduction of these generators to $\Z/p \otimes T'$ is also standard basis. Thus,
for each prime~$p$, $\Z/p \otimes (T'/I)$ is flat over $\Z/p \otimes T =
(\Z/p)[t]_{(t)}$ with the same Hilbert function, so $T'/I$ is flat over~$T$.
\end{remark}

Our next lemma bounds 
the entry~${\it 5} \! + \! {\bf 6}$ in Table~\ref{tbl:isomorphism-counts}.

\begin{lem}\label{lem:almost-monomial}
If a Mustafin triangle $\Gamma$ has two distinct bend points, but is not contained in one apartment,
then the $\M(\Gamma)_k$ is one of the following six ideals:
\begin{align*}
(-2, 2, 0, 1, -1, -2, -1, 0) &\quad
\langle y_3z_1, x_2z_3, y_2z_3, x_1y_3, x_1y_2, y_2z_1,
z_3z_1, x_2z_1, x_3y_2-x_2y_3 \rangle \\
(-3, -1, 4, 5, 4, 1, 0, 1) &\quad
\langle x_2z_3, z_1z_3, y_2z_3, x_1y_3, z_1y_2, x_1y_2, x_1z_3,
y_2x_3-x_2y_3, x_1z_2 \rangle \\
(-1, -3, 1, 4, 2, 2, 0, -1) &\quad
\langle y_1z_3, x_1z_3, x_1y_2, x_3y_1, y_1z_2, z_2x_1,
z_2x_3, z_2z_3, x_2z_3-x_3y_2 \rangle \\
(-3, -4, 3, 2, 3, 0, -1, 0) &\quad
\langle x_1z_3, x_1y_2, x_1y_3, y_3z_2, y_1z_2, z_2x_1,
z_3z_2, y_2y_3-x_2z_3, z_3z_1 \rangle \\
(-3, -1, 2, 1, 2, -1, 0, 0) &\quad
\langle z_1y_2, z_2x_1, y_2z_3, x_1z_3, x_1y_2, y_2y_3,
y_3z_1, x_1y_3, y_3z_2-x_2z_3 \rangle \\
(-2, -4, 3, 2, 3, 1, 0, -1) &\quad
\langle y_1z_3, x_1z_3, x_1y_2, x_1y_3, y_1z_2, y_3z_2,
z_2x_1, z_3z_2, y_2y_3-x_2z_3 \rangle
\end{align*}
\end{lem}

\begin{proof}
By Proposition~\ref{prop:existence-secondarycomp}, $\M(\Gamma)_k$ has five
components.  Lemma~\ref{lem:good-coords} implies that four of these components
are defined by monomial ideals, and the fifth, $C$, is a surface in
$\PP^1_k \times \PP^2_k$. We fix coordinates $(x_1{:}y_1)$ on $\PP^1_k$
and coordinates $(x_2{:}y_2 {:} z_2)$ on $\PP^2_k$. The equation for $C$ has the form
$\begin{bmatrix}x_1 & y_1\end{bmatrix} M \begin{bmatrix} x_2 & y_2 & z_2
\end{bmatrix}^t$, where $M$ is a $2 \times 3$ matrix. If all the non-zero
entries of~$M$ lie in a single row or a single column, then $C$ would be
reducible. Without loss of generality, we assume that the coefficients of both
$x_1x_2$ and $y_1 y_2$ are non-zero. In particular, the point defined by
$x_1 = x_2 = z_2 = 0$ is not in $C$.

Fix any term order with $x_1,x_2 >  y_1, y_2, z_2$.
The initial ideal of~$C$ equals~$\langle x_1x_2 \rangle$. Let $B$
denote the union of the other components. The monomial ideal defining $B$
is its own initial ideal. By \cite[Thm.~2.1]{CS}, the initial ideal  of the special fiber is
radical. Its Hilbert function is sum of the Hilbert functions of $\initial(C)$
and of $\initial(B) = B$ minus that of $\initial(C) \cap
\initial(B)$. However, the Hilbert function is constant under taking an initial
ideal and $C\cap B$ is already a monomial ideal, so $\initial(C) \cap
\initial(B)$ equals $C \cap B$. In particular, $B$ does not contain the
point $x_1 = x_2 = z_2 = 0$. Thus, $\M(\Gamma)_k$ has an
initial monomial ideal in which $C$ degenerates to the union of $\PP_k^2$ and $\PP^1_k
\times \PP^1_k$, whose common line ($x_1=x_2 = 0$) contains a coordinate point
($x_1 = x_2 = z_2 = 0$)
not in any other
component.

When examining the pictures of the $13$ monomial schemes on the main component
of $H_{3,3}$~\cite[Figure~2]{CS}, we find that there are, up to symmetry, $22$
coordinate points on the line between a $\PP^2$ and a $\PP^1 {\times} \PP^1$,
but not on any other component.  Each of these points yields a scheme with $5$
components by replacing the $\PP^2$ and $\PP^1 {\times} \PP^1$ with the blow-up
of $\PP^2$ at a point. However, there are two possible ways of obtaining each
scheme, so there are 11 schemes, {\it 5} of which are planar. The remaining {\bf
6} special fibers, not achievable by a configuration in one apartment, are those
listed in the statement.

The computational methods outlined in Remark~\ref{rmk:compute} show that
the ideals can be realized as the special fiber $\M(\Gamma)_k$ by a
configuration as in Proposition~\ref{prop:monomial-realizable}, where the vector
of integers $(a,b,c,d,e,f,g,h)$ is listed to the left of each ideal. 
\end{proof}

\begin{proof}[Proof of Theorem~\ref{thm:TriangleCensus}]
Bearing in mind the planar classification in Figure \ref{fig:18types},
we enumerate all possibilities based on the number of bent lines.
If that number is zero or one then Lemma~\ref{lem:two-unbent} proves the claim.

Next suppose that $\Gamma$ has two bent lines and one unbent line, say,
between $[L_2]$ and~$[L_3]$. If the two bent lines have the same
bend point $v$, then none of the bend points can lie in $\Gamma$, so
$\M(\Gamma)_k$ has four components. 
 Then $\{v, [L_2], [L_3]\}$ is a configuration with no bent lines,
and thus, by Corollary~\ref{cor:three-unbent}, it must lie in a single
apartment. 
The first lattice point after $v$ on the line
 to either $[L_2]$ or $[L_3]$  corresponds to a line  in
$L_v / \pi L_v$ under the bijection of Lemma~\ref{sect:mus;lem:neighbours}, 
so the angle between the two edges must be either $0$ or~$120$
degrees. Since the line between $[L_2]$ and~$[L_3]$ is unbent, the angle must be
$0$, i.e.\ the three vertices lie on a straight line. Without loss of
generality, we assume that $[L_3]$ lies between
$v$ and~$[L_2]$. Therefore, if we choose an apartment containing $[L_1]$
and~$[L_2]$, then it will contain $[L_3]$ as well.
If $\Gamma$ has two bent lines whose bend points do not coincide, then
$\M(\Gamma)_k$ has five components by
Proposition~\ref{prop:existence-secondarycomp}. By
Lemma~\ref{lem:almost-monomial}, there are six non-planar types
in this case.

What remains to be proved is 
the last column in Table~\ref{tbl:isomorphism-counts}.
The top entry ${\it 1}+{\bf 0}\,$ is correct because here
each of the three lines from $\Gamma$ is bent at the third point,
and the only possibility for this is $\,\Gamma = \,$
\includegraphics[scale=.37]{typ-e3-T.eps} \
with $\,\M(\Gamma)_k =\,$
\includegraphics[scale=.18]{typ-e3-N.eps}.
The fourth entry ${\it 5}+{\bf 8}\,$ in the last column of
Table  \ref{tbl:isomorphism-counts} is correct by Corollary \ref{cor:six-comps}
and Proposition \ref{prop:monomial-realizable}.

Suppose  that $\Gamma$ has three bent lines but there is only one
bend point $v=  [L_0]$. By Proposition~\ref{prop:existence-secondarycomp},
there is only one secondary component, indexed by $v$.
Consider the first step on the line from $v$ to
any of the vertices in~$\Gamma$. The first steps give us one-dimensional subspaces 
of $L_0 / \pi L_0 \simeq k^3$, hence points in $\PP_k^2$. These points must
be distinct or else the convex hull of two lattice points $[L_i]$ and~$[L_j]$ would not pass
through $v$.
 Each $L_i$ can
be chosen to be of the form $M_i + \pi^{t_i} L_0$, where each $M_i$ is
a rank $1$ direct summand of~$L_0$. 
Up to automorphisms of~$\PP^2_k$, there are two possibilities
for three distinct lines in~$\PP_k^2$: either they are collinear or not. In the
latter case, we fix a generator for each of $M_1$, $M_2$ and~$M_3$.
Since
the images of these three elements in $L_0 / \pi L_0$ are linearly independent, they form a
basis for $L_0$ and the configuration lies in the corresponding apartment. Hence
$\,\Gamma = \,$
\includegraphics[scale=.3]{typ-c1-T.eps} \
and $\,\M(\Gamma)_k =\,$
\includegraphics[scale=.18]{typ-c1-N.eps}.
In the former case, we choose generators $e_1'$ and~$e_2'$  of $M_1$ and~$M_2$, respectively, and
let $e_3$ be an element completing these to a basis.
Then $M_1$ is generated by 
$u e_1' + v e_2' + r e_3$, 
where $r$ is in the maximal ideal
and $u$ and $v$ are units by
the assumption that $\overline M_3$ is distinct from the other two reductions.
By substituting $e_1$ and $e_2$ for $u e_1'$ and $v e_2'$ respectively, $M_1$
and $M_2$ remain generated by $e_1$ and $e_2$ respectively, and $M_1$ is
generated by $e_1 + e_2 + r e_3$. It can be
checked that any
configuration of this form leads to the ``sailboat'' of Example~\ref{ex:sailboat}.

We now assume
that $\Gamma$ has three bent lines and exactly two of the bend points coincide. By Proposition~\ref{prop:existence-secondarycomp},
$\Gamma$ has five components. We must show that
(\ref{eqn:airplane}) and
(\ref{eqn:mutant-airplane}) are the only possibilities.
Suppose $v$ is the common bend point of the lines determined by $\{[L_1],[L_2]\}$ and 
$\{[L_1], [L_3]\}$.
Consider the triple $\{v, [L_2], [L_3]\}$. If it lies
in a  single apartment, then either it lies on an unbent line or $v$ is
the bend point between $[L_2]$ and~$[L_3]$, both of which contradict our
assumptions. Thus, $\{v, [L_2], [L_3]\}$  is
  one of the non-planar configurations in
Lemma~\ref{lem:two-unbent}. Since the first step from $v$ to either $[L_2]$
or $[L_3]$ defines a point (and not a line) in $\PP^2_k$, 
the only possibility is case (i), and we have
$\,v = [R\{e_1, e_2, e_3\}] $,
$L_2 = R\{e_1, \pi^s e_2, \pi^s e_3\}$, 
$L_3 = R\{e_1 {+} \pi ^t e_2, \pi^u e_2, \pi^u e_3\}$ with
 $0 {<} t {<} s, u$.

The lattice $L_1$ must have the form $M_1 + \pi^r R\{e_1, e_2, e_3\}$ where
$M_1$ is a rank $1$ direct summand of $R\{e_1, e_2, e_3\}$. Since $v$ is
the bend point of $\{L_1, L_2\}$, the first steps from $v$ to $L_1$ and $L_2$ cannot coincide. 
Hence $\overline M_1 \not= k\{e_1\}$, so 
$M_1$ is generated by $a e_1 + b e_2 + c e_3$, where $b$ and~$c$ do not both
have positive valuation.
If $c$ is a unit, then we can take the change of basis $e_3' = a e_1 +
b e_2 + c e_3$ and then our lattices become
\begin{equation*}
L_1 = R\{ \pi^r e_1, \pi^r e_2, e_3' \}, \,\,\,
L_2 = R\{e_1, \pi^s e_2, \pi^s e_3'\}, \,\,\,
L_3 = R\{e_1 + \pi^t e_2, \pi^u e_2, \pi^u e_3'\}.
\end{equation*}
The corresponding special fiber is (\ref{eqn:airplane}) from Example~\ref{ex:three-bent}. On
the other hand, if $c$ is not a unit,
then $b$ must be a unit, and we take $e_1' = (1-\pi^t
a/b)e_1$ and $e_2' = (a/b) e_1 + e_2$ to get the lattices
\begin{equation*}
L_1 = R \{\pi^r e_1', \pi^r e_3, e_2' + (c/b)e_3\}, \,
L_2 = R\{e_1', \pi^s e_2', \pi^s e_3 \},\,
L_3 = R\{e_1' + \pi^t e_2', \pi^u e_2', \pi^u e_3 \},
\end{equation*}
where $c/b$ has positive valuation.  Now, the special fiber
is~(\ref{eqn:mutant-airplane}) from Example~\ref{ex:three-bent}.
\end{proof}

\bigskip

\noindent \textbf{Acknowledgements.}
This research project was conducted during the
Fall 2009 program on \emph{Tropical Geometry} at
the Mathematical Sciences Research Institute (MSRI) in Berkeley.
Dustin Cartwright and Bernd Sturmfels
were partially supported by the U.S.~National Science
Foundation (DMS-0456960 and DMS-0757207).
Annette Werner and Mathias H\"abich acknowledge support
by the Deutsche Forschungsgemeinschaft (DFG~WE~4279/2-1) and by the Deutscher Akademischer Austausch Dienst (DAAD-grant~D/09/40336). We are indebted to Michael 
Joswig for helping us with  Figure~\ref{fig:exploding}.

\smallskip

\noindent {\bf Authors' addresses:}

\smallskip

\noindent Dustin Cartwright and Bernd Sturmfels, Department of Mathematics,
University of California, Berkeley, CA 94720, USA,
\url{{dustin,bernd}@math.berkeley.edu}

\smallskip

\noindent  Mathias H\"abich and Annette Werner, 
Institut f\"ur Mathematik,
Goethe-Universit\"at, 
%Robert-Mayer-Stra\ss e 8,
60325 Frankfurt am Main, Germany,
\url{{haebich,werner}@math.uni-frankfurt.de}

\end{document}